\documentclass[final]{dmtcs-episciences}

\usepackage[utf8]{inputenc}
\usepackage[round,numbers]{natbib}

\usepackage{amsmath,amssymb,amsthm}
\usepackage{tikz}
\usepackage{subfig}
\usepackage{rotating}
\tikzstyle{vertex}=[circle, draw, inner sep=0pt, minimum size=4pt]
\newcommand{\vertex}{\node[vertex]}

\newcommand{\R}{\mathbb{R}}

\newcommand{\Z}{\mathbb{Z}}

\newcommand{\G}{\mathcal{G}}

\newcommand{\bis}[1]{B(#1,G)}

\newcommand{\PG}{P_G}

\newcommand{\facets}{N}

\newcommand{\Cws}{C_{WS}}

\theoremstyle{plain}
\newtheorem{theorem}{Theorem}[section]
\newtheorem{prop}[theorem]{Proposition}

\newtheorem{conjecture}[theorem]{Conjecture}

\newtheorem{lemma}[theorem]{Lemma}
\newtheorem{corollary}[theorem]{Corollary}

\theoremstyle{definition}

\newtheorem{definition}[theorem]{Definition}

\newtheorem{example}[theorem]{Example}

\author{Benjamin Braun\affiliationmark{1}
  \and Kaitlin Bruegge\affiliationmark{2}
  \and Matthew Kahle\affiliationmark{3}}
\title[Facets of Random Symmetric Edge Polytopes\ldots]{Facets of Random Symmetric Edge Polytopes, Degree Sequences, and Clustering}
\affiliation{
  University of Kentucky, USA\\
  University of Cincinnati, USA\\
  The Ohio State University, USA}

\begin{document}

\publicationdata{vol. 25:2 }{2023}{16}{10.46298/dmtcs.9925}{2022-08-16; 2022-08-16; 2023-05-03; 2023-09-25}{2023-09-26}
\maketitle
\begin{abstract}
Symmetric edge polytopes are lattice polytopes associated with finite simple graphs that are of interest in both theory and applications.
We investigate the facet structure of symmetric edge polytopes for various models of random graphs.
For an Erd\H{o}s-Renyi random graph, we identify a threshold probability at which with high probability the symmetric edge polytope shares many facet-supporting hyperplanes with that of a complete graph.
We also investigate the relationship between the average local clustering, also known as the Watts-Strogatz clustering coefficient, and the number of facets for graphs with either a fixed number of edges or a fixed degree sequence.
We use well-known Markov Chain Monte Carlo sampling methods to generate empirical evidence that for a fixed degree sequence, higher average local clustering in a connected graph corresponds to higher facet numbers in the associated symmetric edge polytope.
\end{abstract}

\section{Introduction}\label{sec:intro}
Given a finite simple graph \(G\), the \emph{symmetric edge polytope} is defined as
\[
\PG := \{\pm(e_i-e_j):ij\in E(G)\} \, .
\]
Symmetric edge polytopes, also called \emph{type PV adjacency polytopes}, have been the subject of extensive recent study~\cite{braunbruegge2022facets,chen2021facets,chen2017counting,chen2020graphadjacency,dali2022gammavector,dalidelucchimichalek,higashitanifanopolytopes,higashitanijochemkomateusz,kalman2022ehrhart,kalman2022h*,matsuietal2011,smoothfanoehrhart2012,osughitsuchiya2020,symmetricedgematchingpolys}.
One problem of interest is to establish relationships between the combinatorial structure of \(G\) and geometric properties of \(\PG\), motivated by applications in both pure and applied contexts.
One major goal is to compute the normalized volume of \(\PG\), where the normalized volume of a lattice polytope \(P\) is volume with respect to the integer lattice in the affine span of \(P\), normalized so that a minimal-volume lattice simplex has volume \(1\).
Symmetric edge polytopes are known to admit regular unimodular triangulations~\cite{higashitanijochemkomateusz}, and thus the normalized volume is the number of maximal simplices in such a triangulation.
Every regular unimodular triangulation of \(\PG\) induces a regular unimodular triangulation of the boundary of \(\PG\), and hence of each facet.
Because \(\PG\) is reflexive, these techniques involving triangulations and facet structure can be used to study normalized volume.
These triangulations have also been used to study Ehrhart-theoretic properties of \(\PG\), e.g.,~\cite{higashitanijochemkomateusz}.

This is important for applications, because the normalized volume of \(\PG\) is an upper bound on the number of non-zero complex solutions to Laurent polynomial systems arising from Kuramoto models of coupled oscillators~\cite{chen2017counting}.
These upper bounds provide information regarding the possible number of synchronization configurations for a Kuramoto model.
As discussed in~\cite{chen2017counting}, these upper bounds also provide stopping criteria for iterative- and homotopy-based solvers for these models.
In private communication, Tianran Chen and Robert Davis stated to the authors that it would be interesting to estimate an upper bound on the number of facets of a symmetric edge polytope in terms of properties of the underlying graph, due to implications regarding the complexity of combinatorial approaches to solving PV-type power-flow equations~\cite{chendavispersonal}.

Even when considering the general setting of all connected graphs on $n$ vertices, the maximum and minimum number of facets of $P_G$ has not yet been established.
Two of the authors~\cite{braunbruegge2022facets} have previously conjectured the following candidates for facet-maximizing and facet-minimizing connected graphs on a fixed number of vertices.
Note that \(K_n\) denotes the complete graph with \(n\) vertices, and \(K_{a,b}\) denotes the complete bipartite graph with shores of having \(a\) and \(b\) vertices, as shown in Figure~\ref{fig: cmp and cmp bipartite}.
Note also that given a finite list of graphs \(G_1,\ldots,G_t\) with one vertex \(v_i\) from each \(G_i\), the wedge of the graphs is obtained by identifying \(v_1,\ldots,v_t\) to a single vertex.
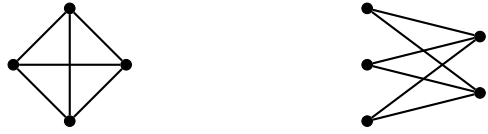
\begin{figure}
    \begin{center}
    \begin{tikzpicture}

\begin{scope}[scale=0.75, xshift=0, yshift=0]
	\vertex[fill](v1) at (0,0) {};
	\vertex[fill](v2) at (-1,1) {};
	\vertex[fill](v3) at (-1,-1) {};	
	\vertex[fill](v4) at (-2,0) {};

	\draw[thick] (v1)--(v2);
	\draw[thick] (v1)--(v3);
	\draw[thick] (v1)--(v4);	
	\draw[thick] (v2)--(v3);
	\draw[thick] (v2)--(v4);
    \draw[thick] (v3)--(v4);
    
    \end{scope}
    
    \begin{scope}[scale=0.75, xshift=150, yshift=0]
	\vertex[fill](v1) at (-1,1) {};
	\vertex[fill](v2) at (-1,0) {};
	\vertex[fill](v3) at (-1,-1) {};
	\vertex[fill](v4) at (1,0.5) {};
	\vertex[fill](v5) at (1,-0.5) {};	

	\draw[thick] (v1)--(v4);
	\draw[thick] (v1)--(v5);
	\draw[thick] (v2)--(v4);
	\draw[thick] (v2)--(v5);	
	\draw[thick] (v3)--(v4);
	\draw[thick] (v3)--(v5);


    \end{scope}
    \end{tikzpicture}
    \end{center}
    \caption{Above are the complete graph \(K_4\) (left), and the complete bipartite graph \(K_{3,2}\) (right).}
    \label{fig: cmp and cmp bipartite}
\end{figure}

\begin{conjecture}[Braun and Bruegge~\cite{braunbruegge2022facets}]\label{conj:maxmin}
    Let $n\geq 3$.
    \begin{enumerate}
        \item For $n=2k+1$, the maximum number of facets for $P_G$ for a connected graph $G$ on $n$ facets is $6^k$, which is attained by a wedge of $k$ cycles of length three.
        \item For $n=2k$, the maximum number of facets for $P_G$ for a connected graph $G$ on $n$ facets is $14\cdot 6^{k-2}$, which is attained by a wedge of $K_4$ with $k-2$ cycles of length three.
        \item For $n=2k+1$, the minimum number of facets for $P_G$ for a connected graph $G$ on $n$ facets is $3\cdot2^k-2$, which is attained by $K_{k,k+1}$.
        \item For $n=2k$, the minimum number of facets for $P_G$ for a connected graph $G$ on $n$ facets is $2^{k+1}-2$, which is attained by $K_{k,k}$.
    \end{enumerate}
\end{conjecture}

Our goal in this work is to study the facets of \(\PG\) for various random graph models, with an emphasis on the number of facets, denoted by \(\facets(\PG)\) or \(\facets(G)\).
First, we investigate properties of the facets of \(\PG\) when \(G\) is an Erd\H{o}s-Renyi random graph, establishing a threshold probability for which \(\PG\) and \(P_{K_n}\) share facet-supporting hyperplanes.
Second, we present the results of empirical investigations regarding the relationship between clustering metrics on \(G\) and \(\facets(G)\).
The graph metric of interest in this work is the \emph{average local clustering coefficient of \(G\)}, also called the \emph{Watts-Strogatz clustering coefficient}.
For each vertex \(v\)  of \(G\), define the local clustering coefficient \(\Cws(v)\) to be the number of edges connecting two neighbors of \(v\) divided by the number of possible edges between neighbors of \(v\).
The average local clustering coefficient is then defined as
\[
\Cws=\frac{1}{|V(G)|}\sum_{v\in V(G)}\Cws(v) \, .
\]
This value is a measure of graph transitivity introduced by Watts and Strogatz~\cite{wattsstrogatz1998} in the context of network science. 

Figure~\ref{fig:8allconnected} contains a plot of average local clustering against number of facets for all $11,117$ connected graphs on $8$ vertices.
There are some apparent patterns in this data.
For example, we have that $N(K_8)=254$ and there are many graphs with $N(G)\approx 254$ having a wide range of clustering values.
Further, there is a general trend that the number of facets increases with the clustering; the slope of the fit line for this data is approximately $148.46$.

Given the rapid rate at which the number of connected graphs on $n$ vertices grows, and the computational expense of computing the facets of a polytope given its vertices~\cite{convexhullcomputations}, it is not productive to attempt to compute facet data for all connected graphs on $n$ vertices in general.
However, in this paper we extend our observations for the $n=8$ data by considering experimental evidence for various families of graphs with a fixed graph invariant. 
These experiments involve generating ensembles of graphs via well-known Markov Chain Monte Carlo techniques.
Specifically, we present experimental evidence that for connected graphs with a fixed degree sequence, higher average local clustering tends to produce symmetric edge polytopes with a larger number of facets than those with lower clustering.
All computations in this work were done with SageMath~\cite{sage} (for graph generation/sampling and clustering metrics) and Normaliz~\cite{normaliz,pyramiddecomp} (for facet computations).
Because it is computationally expensive to compute the number of facets for a polytope, we limit our empirical studies to ensembles of graphs with less than \(20\) vertices; this corresponds to polytopes of dimension at most \(19\).

\begin{figure}
    \centering
    \includegraphics[width=.5\textwidth]{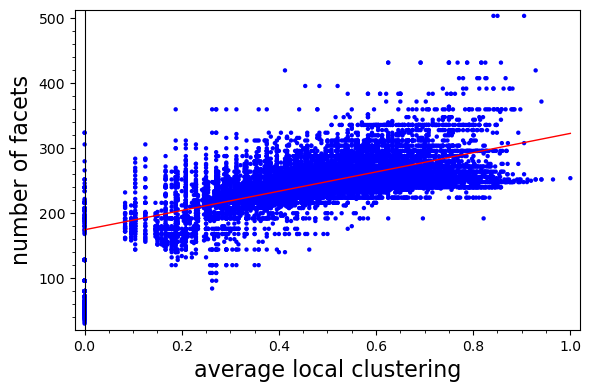}
    \caption{Average local clustering against number of facets for every connected graph \(G\) with 8 vertices. The line of best fit for this data is also included.}
    \label{fig:8allconnected}
\end{figure}

This paper is structured as follows.
In Section~\ref{sec:ERmodel}, we investigate the number of facets for random graphs using the Erd\H{o}s-Renyi model.
We provide empirical data and we prove asymptotic results regarding the existence of certain facets for random graphs.
In Section~\ref{sec:MCMC}, we use Markov Chain Monte Carlo sampling techniques to generate random ensembles of graphs with either a fixed number of edges or a fixed degree sequence, and consider the relationship between average local clustering and facet numbers.
Finally, in Section~\ref{sec:discussion}, we provide a toy example of a theoretical study regarding facet-maximizing graphs with a fixed degree sequence, and we conclude with a general discussion of our results and data.


\section{The Erd\H{o}s-Renyi Model}\label{sec:ERmodel}

A standard model for empirically sampling graphs is the Erd\H{o}s-Renyi model, denoted \(G(n,p)\), for \(n\) a positive integer and \(0<p<1\).
In \(G(n,p)\), each edge occurs independently with probability \(p\); thus, a graph on \(n\) vertices having \(m\) edges occurs with probability \(p^m(1-p)^{\binom{n}{2}-m}\).
Note that for \(G=(V,E)\sim G(n,p)\), meaning a random graph sampled from \(G(n,p)\), the expected value of \(\Cws\) is \(p\).

\subsection{Facets of \(G(n,p)\)}

There are many well-known limitations to the model \(G(n,p)\), most notably that the expected degree distribution, average shortest path length, and various clustering metrics often do not match observations in real-world networks~\cite{newmannetworks}.
However, one of the major advantages of working with \(G(n,p)\) is that it is defined in a way that often allows theoretical results to be obtained.
For example, see~\cite{dali2022gammavector} for the use of the \(G(n,p)\) model in the study of Ehrhart theory for symmetric edge polytopes.
In this subsection, we prove that for certain values of $p$, the typical $G(n,p)$ will have facets of a prescribed form.

An established approach to the facet description of $\PG$ involves certain functions \(f:V\rightarrow \Z\) on the set \(V\) of vertices in \(G\). 
Specifically, it was shown in~\cite[Theorem 3.1]{higashitanijochemkomateusz} that the facets of \(\PG\) arise as follows.

\begin{theorem}[Higashitani, Jochemko, Michalek \cite{higashitanijochemkomateusz}]\label{thm:facetdescription}
Let \(G=(V,E)\) be a finite simple connected graph.
Then \(f:V\rightarrow\Z\) is facet-defining if and only if both of the following hold.
\begin{enumerate}
    \item[(i)]For any edge \(e=uv\) we have \(|f(u)-f(v)|\leq 1\).
    \item[(ii)] The subset of edges \(E_f=\{e=uv\in E\::\: |f(u)-f(v)|= 1\}\) forms a spanning connected subgraph of \(G\).
\end{enumerate}
\end{theorem}

Symmetric edge polytopes are contained in the hyperplane orthogonal to the span of the all-ones vector, and thus two facet-defining functions are identified if they differ by a constant vector.
The spanning connected subgraphs with edge sets \(E_f\) arising in Theorem~\ref{thm:facetdescription}, called \emph{facet subgraphs}, have further structure. 

\begin{definition}
Given a facet-defining function \(f\) as in Theorem~\ref{thm:facetdescription}, we write \(G_f\) for the corresponding facet subgraph, i.e., the subgraph with vertex set \(V\) and edge set \(E_f\).
\end{definition}

\begin{lemma}[Chen, Davis, Korchevskaia \cite{chen2021facets}]\label{lem:facetsubgraphs}
Let \(G\) be a connected graph. 
A subgraph \(H\) of \(G\) is a facet subgraph of \(G\) if and only if it is a maximal connected spanning bipartite subgraph of \(G\).
\end{lemma}

It is therefore natural to study the expected structure of induced bipartite subgraphs of \(G(n,p)\) across various bipartitions of the verex set.

\begin{definition}
For any bipartition \((A,V\setminus A)\) of the vertex set \(V\) of a graph \(G\), we denote by \(\bis{A}\) the induced bipartite subgraph for the bipartition \((A,V\setminus A)\).
\end{definition}

Note that given any bipartition \((A,V\setminus A)\) of the vertex set of \(G\), if \(\bis{A}\) is connected then there are at least two facets of \(\PG\) having \(G_f=\bis{A}\), specifically the two \(\{0,1\}\)-labelings of the bipartition.
When \(G=K_n\), any bipartition produces a facet subgraph and these \(\{0,1\}\)-labelings are the only facet-supporting functions.

Thus, we are interested in understanding when \(\bis{A}\) is connected for an typical \(G\).
For the following theorem, recall that a sequence of events \(\mathcal{A}_n\) for \(n=1,2,\dots\) occurs \emph{with high probability} (abbreviated w.h.p.) if \(\lim_{n\rightarrow\infty}Prob(\mathcal{A}_n)=1\).

\begin{theorem}\label{thm:ERbound}
Let \(G=(V,E)\sim G(n,p)\).
\begin{itemize}
\item If \(p<1/2\) is fixed, then w.h.p.~there exists an \(\lfloor n/2\rfloor\)-subset \(A\) of \(V\) such that \(\bis{A}\) is not connected.
\item If \(p>1/2\) is fixed, then w.h.p. for every subset \(A\subset V\), \(\bis{A}\) consists of a single connected component unioned with isolated vertices.
\item Further, if \(p=1/2+\epsilon\) is fixed for some \(\epsilon>0\), then w.h.p.~for every subset \(A\subset V\) with \(||A|-n/2|<\epsilon(1/2-\epsilon)n\) we have that \(\bis{A}\) is connected and spans \(V\).
\end{itemize}
\end{theorem}

The proof of Theorem~\ref{thm:ERbound} will require the following well-known lemma.
\begin{lemma} \label{lem:concentration}
Let \(p \in (0,1)\) be fixed, and \(G \sim G(n,p)\). Then w.h.p.\ every vertex in \(G\) has degree \(\approx pn\).
That is, for every fixed \(\epsilon > 0\), w.h.p. every vertex \(v\) has degree \((1-\epsilon) pn < deg(v) < (1 + \epsilon) pn\).
\end{lemma}

\begin{proof}[of Theorem~\ref{thm:ERbound}]
For the case where \(p<1/2\), note that w.h.p. every vertex has degree close to its mean \(pn\). Since \(p < 1/2\), a typical vertex \(v\) is connected to fewer than half of the other vertices.
Thus, there is some \(\lfloor n/2\rfloor\)-subset \(A\) of \(V\) containing the entire neighborhood of \(v\), and the corresponding \(\bis{A}\) is not connected.

Next, let \(p=1/2 +\epsilon\) for fixed \(\epsilon>0\). 
First, we show that for any subset \(A\subset V\), there exists a constant \(\alpha\in \R_{>0}\) depending only on \(\epsilon\) such that \(\bis{A}\) has no connected component of order \(i\) for \(2\leq i\leq \alpha n\).
Let \(\delta=\delta(G)\) denote the minimum degree of a vertex in \(G\).
By Lemma~\ref{lem:concentration}, w.h.p. \(\delta\geq \frac{(1+\epsilon)}{2}n-\epsilon^2n\).
For the bipartition \((A, V\setminus A)\), one of \(A\) and \(V\setminus A\) has order less than or equal to \(n/2\). 
Without loss of generality, suppose \(|A|\leq n/2\).
If \(S\subset V\) spans a connected component of \(\bis{A}\) and \(|S|\geq 2\), then there exists a vertex \(v\) of \(S\) such that \(v\in A\).
Then w.h.p. \(\deg(v)\geq\frac{(1+\epsilon)}{2}n-\epsilon^2n\) and the number of neighbors of \(v\) that are in \(V\setminus A\) must be at least
\[
\begin{split}
\delta-|A|&\geq \frac{(1+\epsilon)}{2}n-\epsilon^2n - n/2  \\
&=\frac{\epsilon n}{2}-\epsilon^2n\\
&=\epsilon(1/2-\epsilon)n.
\end{split}
\]
Setting \(\alpha=\frac{\epsilon}{2}(1/2-\epsilon)\), we see that w.h.p. the order of \(S\) is strictly greater than \(\alpha n\), a contradiction.

Next we rule out connected components of order \(i\) for \(\alpha n\leq i\leq 2n/3\) by bounding the probability that a given subset of \(i\) vertices spans a connected component for some bipartition \((A,V\setminus A)\).
Suppose that \(S\subset V\) with \(|S|=i\) spans a connected component of \(\bis{A}\), and let \(a=|S\cap A|\) and \(b=|S\cap (V\setminus A)|\).
Any edge between \(S\cap A\) and \((V\setminus A)\setminus S\) is forbidden.
If there was such an edge, \(S\) would not span its connected component, contradicting our assumption.
The same holds for edges between \(A\setminus S\) and \(S\cap (V\setminus A)\).
So the number of forbidden edges is \(a(n-b)+b(n-a)=(a+b)n-2ab\), and the probability that \(S\) spans its component in the given bipartition is at most
\[
(1-p)^{(a+b)n-2ab}.
\]
Since \(a+b=i\), \(ab\) is maximized when \(a=\lfloor i/2\rfloor\) and \(b=\lceil i/2\rceil\).
Thus, the number of forbidden edges is at least 
\[
in-2(i/2)^2=i(n-i/2)\geq (\alpha n)(n-n/3)=\frac{2\alpha}{3} n^2.
\]

Applying a union bound, the probability that there is any connected component of size \(\alpha n\leq i \leq \frac{2n}{3}\) in any bipartition is bounded above by 
\[
\sum_{k=1}^{n-1}\binom{n}{k}\sum_{i=\alpha n}^{2n/3}\binom{n}{i}(1-p)^{\frac{2\alpha}{3} n^2}\leq 2^{n}\cdot 2^{n}(1-p)^{\frac{2\alpha}{3} n^2}\rightarrow 0
\]
as \(n\rightarrow \infty\) since \(p\in (0,1)\) is fixed. 
So w.h.p. \(\bis{A}\) consists of a single large component of order at least \(2n/3\) unioned with isolated vertices.

Now we show that if \(||A|-n/2|< \epsilon(1/2-\epsilon)n\), then w.h.p. \(\bis{A}\) has no isolated vertices. 
We know that for \(p=1/2 + \epsilon\) we have \(\delta\geq\frac{(1+\epsilon)}{2}n-\epsilon^2n\).
For a bipartition \((A,V\setminus A)\), one of the shores has size greater than \(n/2\); we have already seen that if \(|A|<n/2\), w.h.p. there are no isolated vertices in \(\bis{A}\) contained in \(A\).
Thus, we can assume that \(|A|>n/2\) and assume further that \(|A|-n/2< \epsilon(1/2-\epsilon)n\).
It follows that
\[
|A|<\frac{(1+\epsilon)}{2}n-\epsilon^2n=\delta
\]
and thus \(\delta-|A|>0\).
Hence, w.h.p. every vertex in \(A\) must have a neighbor in \(V\setminus A\) and thus there are no isolated vertices in \(\bis{A}\).
\end{proof}

Observe that every facet subgraph must support the two facet-defining functions taking values in \(\{0,1\}\), namely those two $0/1$-functions that are constant on the shores of the bipartition induced by $G_f$.
Further, any $0/1$-function on the vertices of $G$ is a facet-defining function for the symmetric edge polytope of the complete graph on those vertices.
Thus, Theorem~\ref{thm:ERbound} shows that when \(p>1/2\) and \(n\) is large, there are many symmetric pairs of facet-supporting hyperplanes that are identical for \(G\sim G(n,p)\) and \(K_n\).
The corresponding facets might not have the same polyhedral structure in the two symmetric edge polytopes, but the facet-supporting hyperplanes are the same.
This leads to the following corollary.

\begin{corollary}
\label{cor:facetbound}
Let \(G=(V,E)\sim G(n,p)\).
For fixed \(p=1/2+\epsilon\) with \(\epsilon>0\), let $t(n,p)$ denote the number of subsets $A\subset V$ satisfying \(||A|-n/2|<\epsilon(1/2-\epsilon)n\).
Then w.h.p. we have that the number of facets of $\PG$ is at least $t(n,p)$, i.e., as $n\to \infty$ the probability that $G\sim G(n,p)$ yields this number of facets goes to $1$.
\end{corollary}

\begin{example}
    Figure~\ref{fig:P_3 and K_3 SEPs} shows two graphs, a path with 3 vertices and a complete graph with 3 vertices, and their symmetric edge polytopes. 
    The bold pairs of facets are supported by the same pair of hyperplanes, in particular, the hyperplanes \(x_v=1\) and \(x_u+x_w=1\), arising from the partition \((\{u,w\}, \{v\})\) of the vertices.
    Since the path graph does not contain the edge \(uw\), it is not necessary that the facet-defining functions be constant on the shores of this partition.
    Therefore, this same partition also determines the other pair of facets for the path graph, supported by the hyperplanes \(x_u-x_w=1\) and \(x_w-x_u=1\), which are not support hyperplanes for \(K_3\).
\end{example}

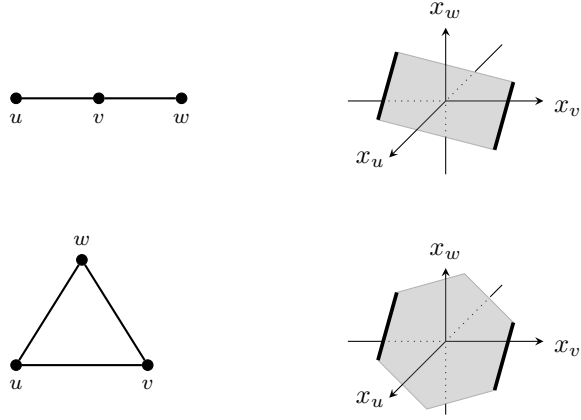
\begin{figure}[h!]
    \centering

\begin{tikzpicture}
 
    \begin{scope}[scale=0.55, xshift=0, yshift=120]
	\vertex[fill,label=below:\footnotesize{$u$}](a1) at (0,0) {};
	\vertex[fill,label=below:\footnotesize{$v$}](a2) at (2,0) {};
	\vertex[fill,label=below:\footnotesize{$w$}](a3) at (4,0) {};

	\draw[thick] (a1)--(a2);
	\draw[thick] (a2)--(a3);

    \end{scope}

 \begin{scope}[scale=0.65, xshift=250, yshift=100]
\draw[thin] (0,0,-3)--(0,0,-1.6); 
\draw[thin,dotted] (0,0,-1.6)--(0,0,0); 
\draw[thin, -stealth] (0,0,0)--(0,0,3); 

\draw[thin] (0,-1.5,0)--(0,-.75,0);
\draw[thin, dotted] (0,-.75,0)--(0,0);
\draw[thin, -stealth] (0,0,0)--(0,1.5,0);

\draw[thin] (-2,0,0)--(-1.3,0,0);
\draw[thin, dotted] (-1.3,0,0)--(0,0,0);
\draw[thin, -stealth] (0,0,0)--(2,0,0);

\coordinate[label=$x_u$] (xu) at (0,0,4);
\coordinate[label=$x_v$] (xv) at (2.5,-0.5,0);
\coordinate[label=$x_w$] (xw) at (0,1.5,0);

 \coordinate[] (u-v) at (-1,0,1);
 \coordinate[] (v-u) at (1,0,-1);
 \coordinate[] (v-w) at (1,-1,0);
 \coordinate[] (w-v) at (-1,1,0);

  \draw[fill=gray,opacity=0.3] (u-v)--(w-v)--(v-u)--(v-w)--(u-v);
  \draw[line width=1.4pt, opacity=1] (u-v)--(w-v);
  \draw[line width=1.4pt, opacity=1] (v-u)--(v-w);
  
\end{scope}

    \begin{scope}[scale=0.35, xshift=0, yshift=-100]
	\vertex[fill,label=below:\footnotesize{$u$}](a1) at (0,0) {};
	\vertex[fill,label=below:\footnotesize{$v$}](a2) at (5,0) {};
	\vertex[fill,label=above:\footnotesize{$w$}](a3) at (2.5,4) {};

	\draw[thick] (a1)--(a2);
	\draw[thick] (a2)--(a3);
	\draw[thick] (a1)--(a3);

    \end{scope}

 \begin{scope}[scale=0.65, xshift=250, yshift=-40]
\draw[thin] (0,0,-3)--(0,0,-2.4); 
\draw[thin,dotted] (0,0,-2.4)--(0,0,0); 
\draw[thin, -stealth] (0,0,0)--(0,0,3); 

\draw[thin] (0,-1.5,0)--(0,-1.3,0);
\draw[thin, dotted] (0,-1.3,0)--(0,0);
\draw[thin, -stealth] (0,0,0)--(0,1.5,0);

\draw[thin] (-2,0,0)--(-1.3,0,0);
\draw[thin, dotted] (-1.3,0,0)--(0,0,0);
\draw[thin, -stealth] (0,0,0)--(2,0,0);

\coordinate[label=$x_u$] (xu) at (0,0,4);
\coordinate[label=$x_v$] (xv) at (2.5,-0.5,0);
\coordinate[label=$x_w$] (xw) at (0,1.5,0);

 \coordinate[] (u-v) at (-1,0,1);
 \coordinate[] (v-u) at (1,0,-1);
 \coordinate[] (v-w) at (1,-1,0);
 \coordinate[] (w-v) at (-1,1,0);
 \coordinate[] (u-w) at (0,-1,1);
 \coordinate[] (w-u) at (0,1,-1);

  \draw[fill=gray,opacity=0.3] (u-v)--(w-v)--(w-u)--(v-u)--(v-w)--(u-w)--(u-v);
    \draw[line width=1.4pt, opacity=1] (u-v)--(w-v);
  \draw[line width=1.4pt, opacity=1] (v-u)--(v-w);
\end{scope}

\end{tikzpicture}
    
    \caption{A path graph \(P_3\) (top) and the complete graph \(K_3\) (bottom) with their symmetric edge polytopes. Facets arising from the same pair of support hyperplanes are bold.}
    \label{fig:P_3 and K_3 SEPs}
\end{figure}

The next corollary shows that with high probability, there is at least one \(G_f\) in \(G\sim G(n,1/2+\epsilon)\) that supports for \(\PG\) only the facet-supporting hyperplanes that the same bipartition supports for \(P_{K_n}\).

\begin{corollary}
\label{cor:Gfcentral01}
Let \(G_n=(V_n,E)\sim G(n,p)\).
If \(p=1/2+\epsilon\) is fixed with \(\epsilon>0\), then for a sequence of subsets \(A_n\subset V_n\) with \(||A_n|-n/2|<\epsilon(1/2-\epsilon)n\), w.h.p. there are only two facet-defining functions \(f\) of \(P_{G_n}\) with \(G_f=\bis{A_n}\), namely the functions \(f\) with values in \(\{0,1\}\).
\end{corollary}

\begin{proof}
First note that as \(n\to \infty\), w.h.p. the induced subgraph of \(G\) on a given subset of vertices \(A\) with \(|A|>((1+\epsilon)/2-\epsilon^2)n\) is connected.
To see this, note that since \(\epsilon\) is fixed, we have \(((1+\epsilon)/2-\epsilon^2)n\) goes to infinity and thus \(p=1/2+\epsilon>\log(((1+\epsilon)/2-\epsilon^2)n)/((1+\epsilon)/2-\epsilon^2)n\).
Thus, the induced subgraph is \(G(((1+\epsilon)/2-\epsilon^2)n,p)\) and, by a well-known theorem of Erd\H{o}s and Renyi~\cite{erdosrenyi1960}, w.h.p. this is connected.
Since both \(\bis{A}\) and \(\bis{V\setminus A}\) are connected w.h.p. for any \(A\) with \(||A|-n/2|<\epsilon(1/2-\epsilon)n\), it follows from Theorem~\ref{thm:facetdescription} that any \(f\) with \(G_f=\bis{A}\) is constant on both \(A\) and \(V\setminus A\).
\end{proof}

\begin{example}
    The graph \(G\) in Figure~\ref{fig:K4 minus an edge} is the subgraph of the complete graph \(K_4\) obtained by removing the edge \(vw\). 
    The bipartition \((\{t,\, v\}, \{u,\,w\})\) of the vertices of \(G\) produces the facet subgraph \(G_f\). 
    Since the edges \(tv\) and \(uw\) are present in the original graph \(G\), Theorem~\ref{thm:facetdescription} gives that any facet-defining function must be constant on the shores of this bipartition. 
    In particular, the only facet-supporting hyperplanes of \(\PG\) that arise from this partition are \(x_t+x_v=1\) and \(x_u+x_w=1\), which are exactly the facet-supporting hyperplanes for \(P_{K_4}\) associated with the same partition.
\end{example}

\begin{figure}[h]
    \centering
    \begin{tikzpicture}

\begin{scope}[scale=1, xshift=0, yshift=0]
	\vertex[fill, label=$t$](v1) at (0,0) {};
	\vertex[fill, label=$u$](v2) at (-1,1) {};
	\vertex[fill, label=below:$v$](v3) at (-1,-1) {};	
	\vertex[fill, label=$w$](v4) at (-2,0) {};

	\draw[thick] (v1)--(v2);
	\draw[thick] (v1)--(v3);
	\draw[thick] (v1)--(v4);	
	\draw[thick] (v2)--(v3);
	\draw[thick] (v2)--(v4);

 \node[] at (-2,1.5) {$G$};
    
    \end{scope}

\begin{scope}[scale=1, xshift=100, yshift=0]
	\vertex[fill, label=$t$](v1) at (0,0) {};
	\vertex[fill, label=$u$](v2) at (-1,1) {};
	\vertex[fill, label=below:$v$](v3) at (-1,-1) {};	
	\vertex[fill, label=$w$](v4) at (-2,0) {};

	\draw[thick] (v1)--(v2);
	\draw[thick] (v1)--(v4);	
	\draw[thick] (v2)--(v3);

     \node[] at (-2,1.5) {$G_f$};
    
    \end{scope}
\end{tikzpicture}
    \caption{The graph \(G\) is formed by removing the edge \(vw\) from the complete graph \(K_4\). The graph \(G_f\) is the facet subgraph of \(G\) associated to the partition \((\{t,\, v\}, \{u,\,w\})\) of the vertices.}
    \label{fig:K4 minus an edge}
\end{figure}
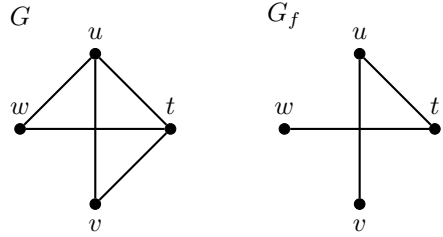

\subsection{Data and Observations}
While Theorem~\ref{thm:ERbound} establishes the existence of many facet subgraphs in certain large random graphs, this does not provide much insight into our consideration of average local clustering.
When we direct our attention to empirical data for \(G(n,p)\), no apparent correlation between average local clustering and \(\facets(G)\) is observed, as demonstrated in Figure~\ref{fig: 14vert gnp 5000} for an ensemble sampled from \(G(14,0.45)\).
This is in stark contrast to the data for all connected graphs on eight vertices shown in Figure~\ref{fig:8allconnected}.
One caveat is that there are $29,003,487,462,848,061$ connected graphs on 14 vertices~\cite[A001349]{OEIS}, of which we sample only 4975.
Figure~\ref{fig:table of plots} provides further evidence that sampling from $G(n,p)$ does not yield a consistent trend.
This figure contains a table of plots providing data for ensembles of graphs with \(n=11,14,17\) and \(p=0.2, 0.4, 0.6, 0.8\). 

\begin{figure}[ht]
    \centering
    \includegraphics[width=0.5\textwidth]{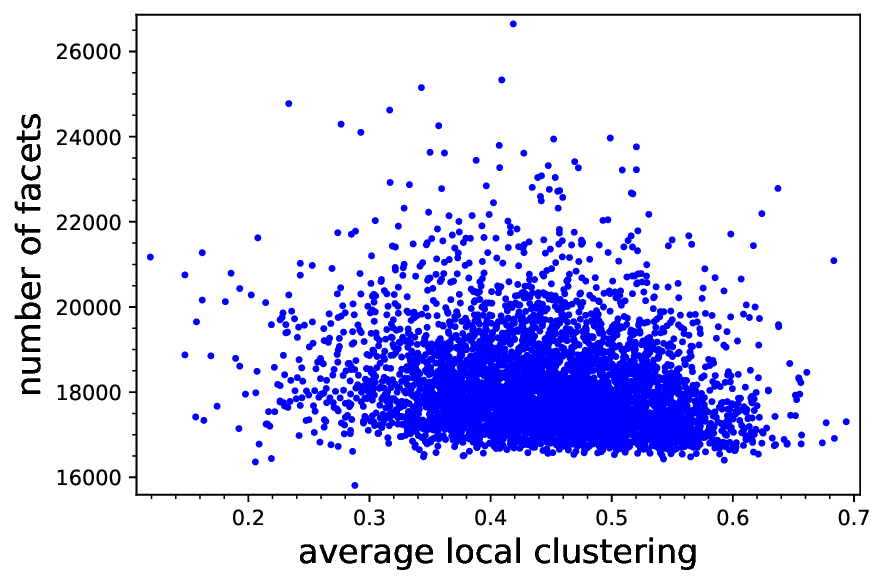}
    \caption{Data from an ensemble of 4975 connected graphs from \(G(14,0.45)\).}
    \label{fig: 14vert gnp 5000}
\end{figure}

\begin{figure}
\centering
\subfloat[$n=11$, $p=0.2$]{ \includegraphics[scale=0.3]{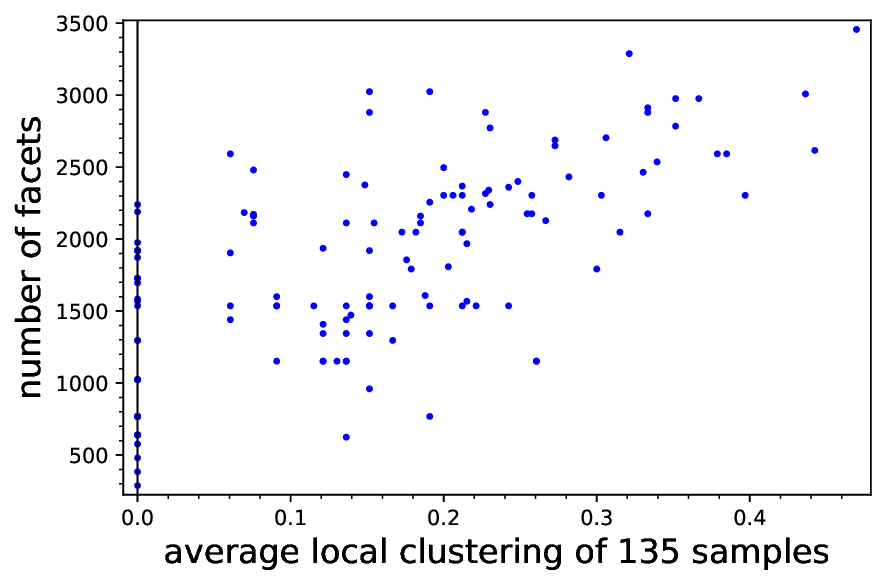}}\hfill
\subfloat[$n=14$, $p=0.2$]{\includegraphics[scale=0.3]{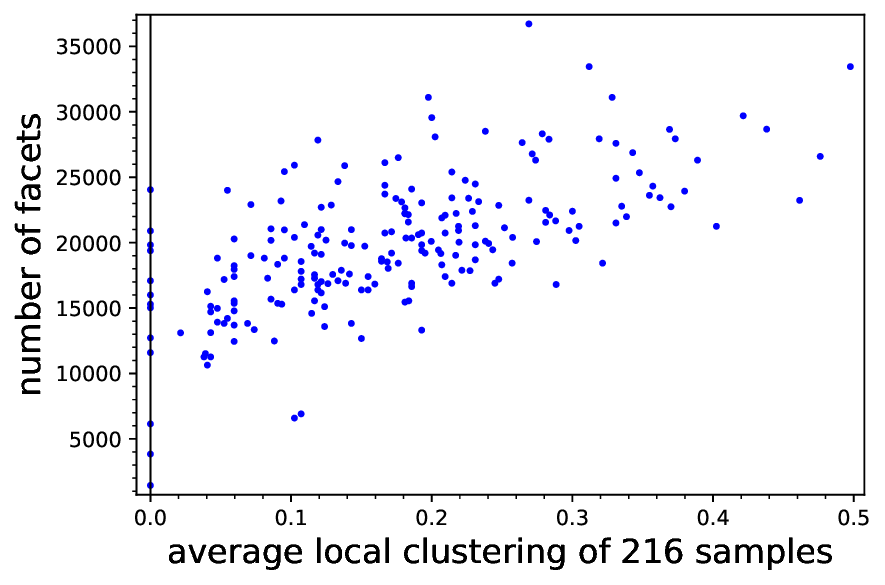}}\hfill
\subfloat[$n=17$, $p=0.2$]{\includegraphics[scale=0.3]{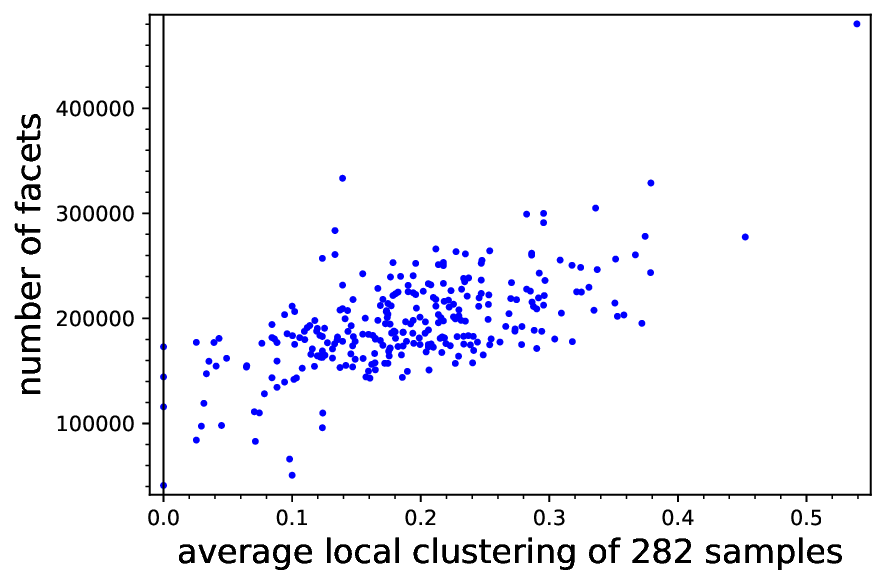}}\\

\subfloat[$n=11$, $p=0.4$]{\includegraphics[scale=0.3]{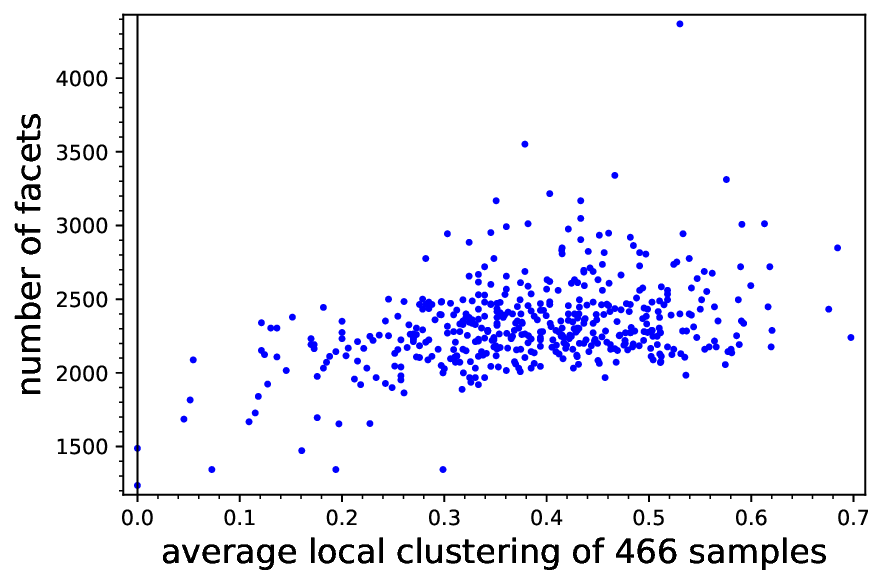}}\hfill
\subfloat[$n=14$, $p=0.4$]{\includegraphics[scale=0.3]{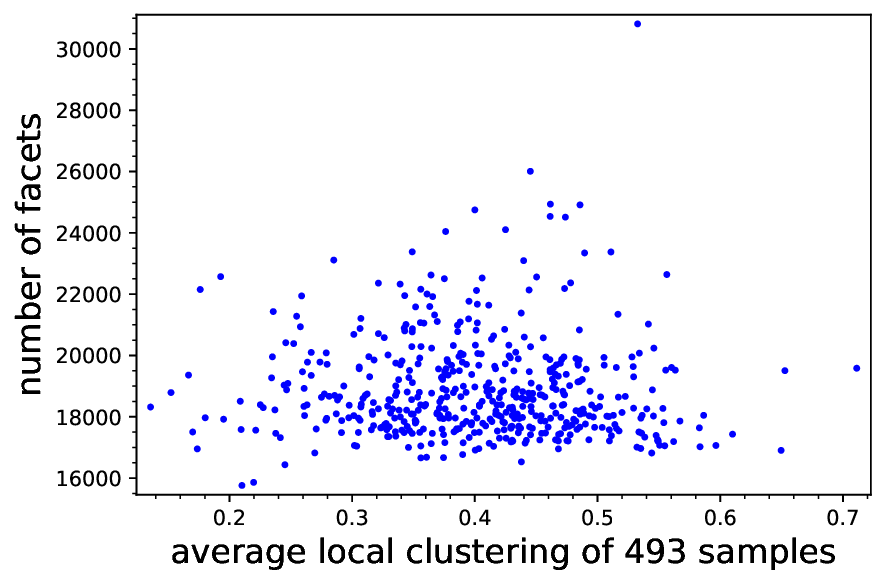}}\hfill
\subfloat[$n=17$, $p=0.4$]{\includegraphics[scale=0.3]{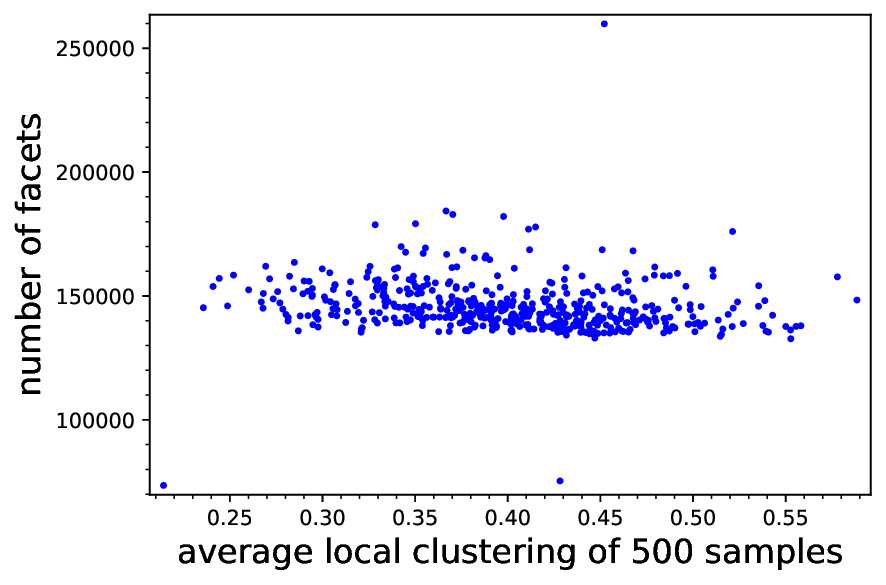}}\\

\subfloat[$n=11$, $p=0.6$]{\includegraphics[scale=0.3]{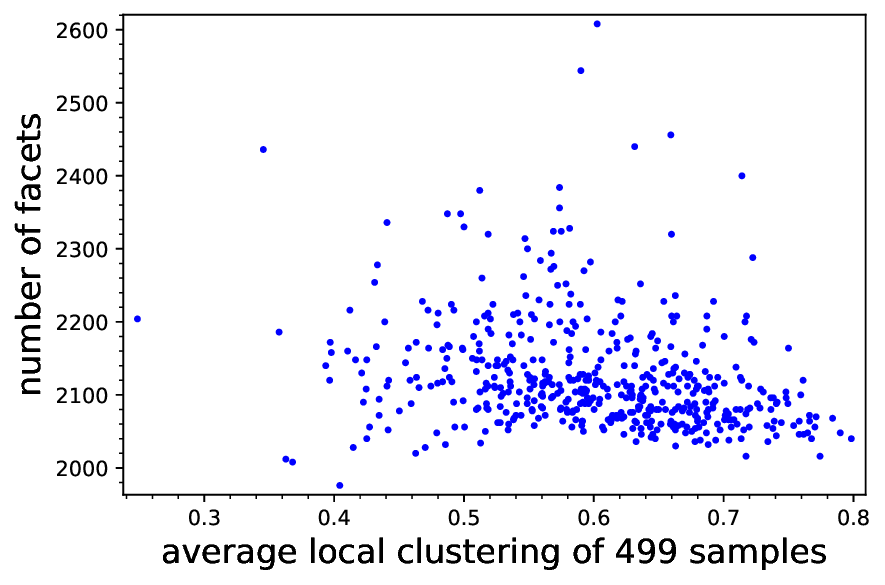}}\hfill
\subfloat[$n=14$, $p=0.6$]{\includegraphics[scale=0.3]{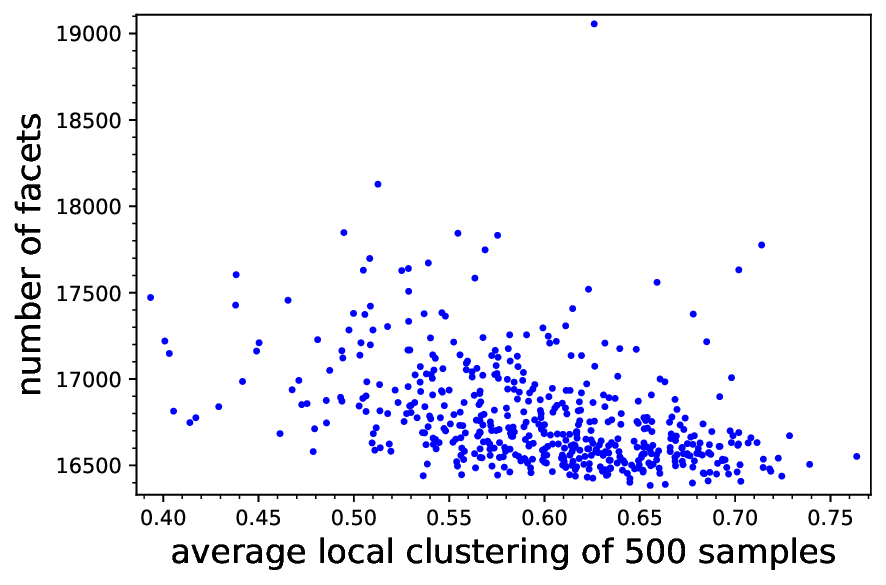}}\hfill
\subfloat[$n=17$, $p=0.6$]{\includegraphics[scale=0.3]{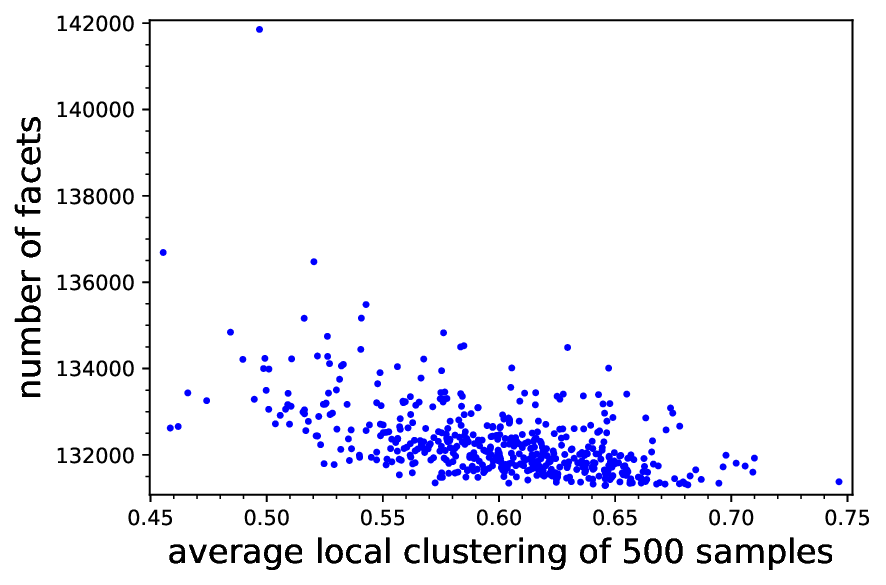}}\\

\subfloat[$n=11$, $p=0.8$]{\includegraphics[scale=0.3]{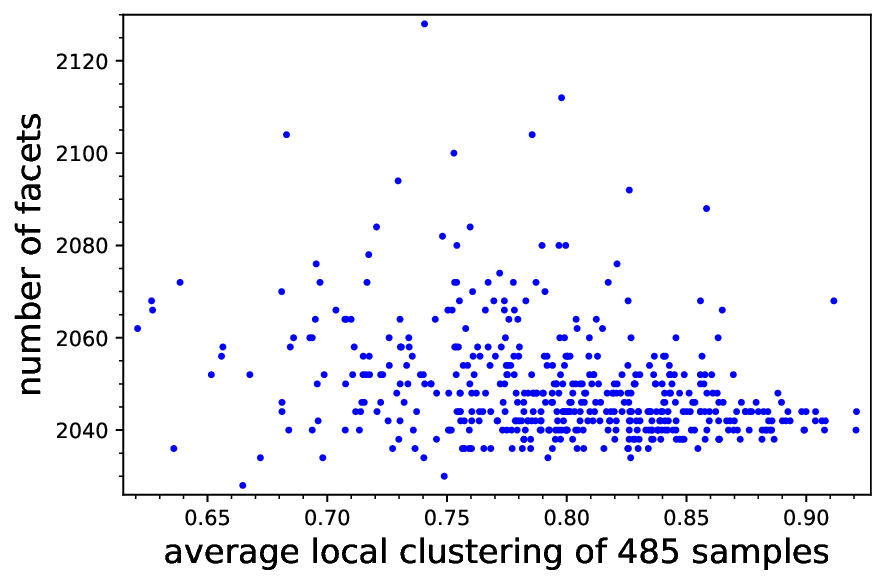}}\hfill
\subfloat[$n=14$, $p=0.8$]{\includegraphics[scale=0.3]{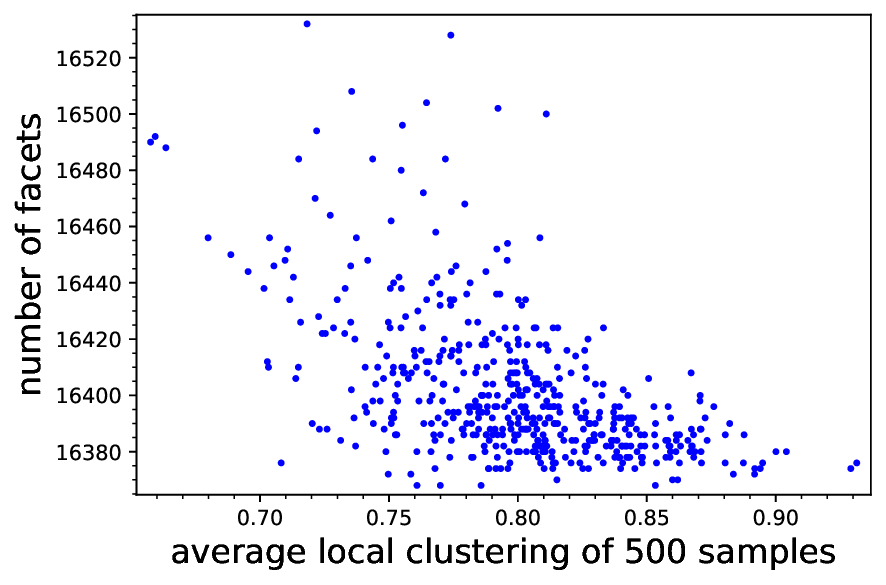}}\hfill
\subfloat[$n=17$, $p=0.8$]{\includegraphics[scale=0.3]{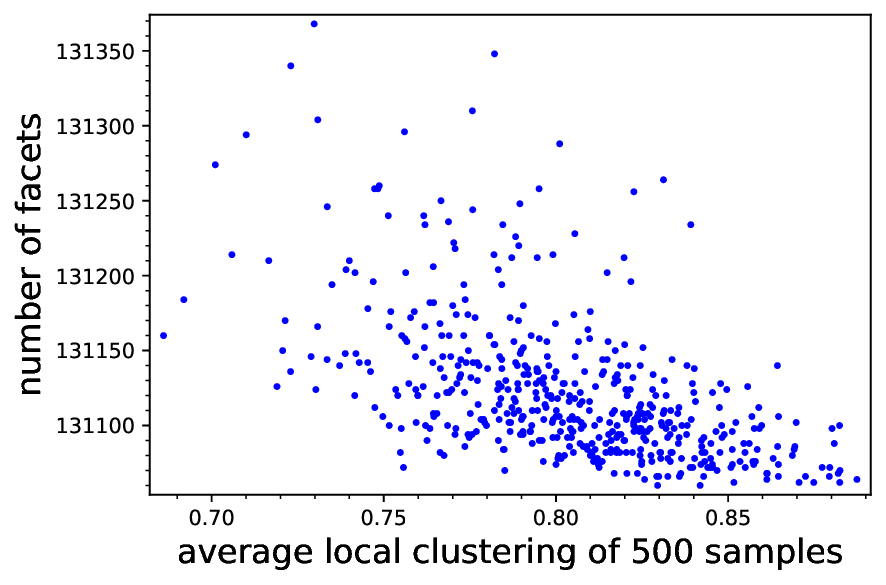}}\\

    \caption{Data from ensembles drawn from \(G(n,p)\).
    The target sample size in each ensemble was 500 connected graphs and disconnected graphs were rejected during sampling.
    Note that for smaller values of \(p\), the range of the vertical axis is significantly larger than for large values of \(p\).}
    \label{fig:table of plots}
\end{figure}

In Figure~\ref{fig:table of plots}, the number of vertices increases across rows while the value of \(p\) increases down columns.
As expected, the mean value of sampled \(\Cws\) values is approximately \(p\).
Further, reading down each column as \(p\) increases, we observe the range of values of \(\facets(G)\) in our sample becomes smaller. 
Specifically, these values are getting closer to \(\facets(K_n)=2^n-2\), where this formula is a straightforward application of Theorem~\ref{thm:facetdescription}. 
This makes sense, as higher values of $p$ leads to greater edge density, and thus the typical sampled graph is closer in structure to $K_n$.

What we do not see consistently in these samples from $G(n,p)$ is a positive correlation between average local clustering and number of facets that is observed in Figure~\ref{fig:8allconnected}.
For $p=0.2$, the first row in Figure~\ref{fig:table of plots}, we do observe a positive trend.
We also observe that when $p=0.2$, there is significantly more variation in the number of facets that arise; note that for $n=17$, the range of the vertical axis when $p=0.2$ is from less than $100,000$ to over $400,000$.
However, for $n=17$ and $p=0.8$, the range is from around $131,000$ to around $131,400$.
Note that $N(K_{17})=131,070$.

What we will see in Section~\ref{sec:MCMC} is that when we sample graphs on $n$ vertices in a more restrictive fashion, fixing also the number of edges or (more strongly) the degree sequence, we do observe a positive correlation between average local clustering and number of facets.

\section{Graph Ensembles via Markov Chain Monte Carlo Sampling}\label{sec:MCMC}

 While our samples from \(G(n,p)\) for fixed \(p\) do not display the relationship between average local clustering and \(\facets(G)\) that was observed in the complete enumeration for small \(n\) (as in Figure~\ref{fig:8allconnected}), we do see correlations when we use other random graph models that differently restrict the space of graphs we consider.
This portion of our study uses Markov Chain Monte Carlo (MCMC) techniques to generate ensembles of connected graphs having either a fixed number of edges or a fixed degree sequence.
These techniques arise in the study of configuration models for random graphs with a fixed degree sequence, see the survey~\cite{SIAMMCMC} and the references given there for more details.

\subsection{MCMC Sampling Methods}\label{subsec:sample_edge_only}

Sampling from graph spaces via Markov chain traversal is a common technique~\cite{SIAMMCMC}. 
For a Markov chain having a certain stationary distribution, sample graphs taken at sufficiently spaced intervals can be treated as independent, and an ensemble of such graphs can be expected to follow the stationary distribution. 
In the case of our study, we employ Markov chains for which the stationary distribution is uniform arising from processes to produce new graphs from old by local, reversible operations. 
With this, we can picture the sample space as a graph of graphs \(\mathcal{G}\), where each node represents a graph in the space, and there is a directed edge from the graph \(G\) to the graph \(G'\) if performing an instance of the transition operation on \(G\) produces \(G'\). 
In general, each edge has a weight signifying the probability of that transition. 
In our study, all instances of the transition operation are equally likely. 
That is, from a state, \(G\), in the Markov chain, the probability of transitioning to any adjacent state is equal. 
So we can view edges in the graph of graphs for our spaces as unweighted.

\subsubsection{Fixed Number of Edges}

To sample from the space of connected graphs with \(n\) vertices and \(m\) edges, we employ a single-edge replacement MCMC technique similar to what is described in~\cite[Section 2]{SIAMMCMC} (in the case of trees, this technique is known as branch exchange).
Define the graph of graphs \(\G(n,m)\) to be the directed graph with vertex set all connected graphs with \(n\) vertices and \(m\) edges.
A connected graph \(G\) has an arrow in \(\G(n,m)\) to a connected graph \(G'\) if \(G'\) is obtained from \(G\) by deleting an edge in \(G\) and adding an edge in \(G'\) from the complement of \(G\). 
This process is demonstrated in Figure~\ref{fig:single edge replacement}.

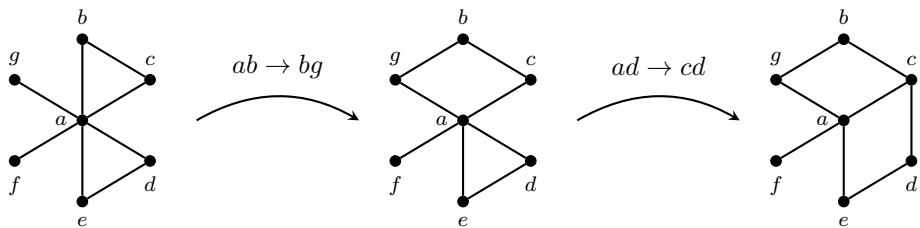
\begin{figure}[h]
    \centering

 \begin{tikzpicture}[scale=0.9]
    
    \begin{scope}[scale=.8, xshift=0, yshift=0]
	\vertex[fill, label=left:\footnotesize{$a$}](v1) at (0,0) {};
	\vertex[fill, label=above:\footnotesize{$b$}](v2) at (0,1.5) {};
	\vertex[fill, label=above:\footnotesize{$c$}](v3) at (1.25,.75) {};
	\vertex[fill, label=below:\footnotesize{$d$}](v4) at (1.25,-.75) {};
	\vertex[fill, label=below:\footnotesize{$e$}](v5) at (0,-1.5) {};	
	\vertex[fill, label=below:\footnotesize{$f$}](v6) at (-1.25,-.75) {};
	\vertex[fill, label=above:\footnotesize{$g$}](v7) at (-1.25,.75) {};

	\draw[thick] (v1)--(v2);
	\draw[thick] (v1)--(v3);
	\draw[thick] (v1)--(v4);
	\draw[thick] (v1)--(v5);	
	\draw[thick] (v1)--(v6);
	\draw[thick] (v1)--(v7);
	\draw[thick] (v2)--(v3);
	\draw[thick] (v4)--(v5);
  
    \end{scope}

    \begin{scope}[scale=.8,xshift=60, yshift=0]
     \draw[thick, -stealth] (0,0) to[out=30,in=150] (3,0);
          \node[] at (1.5,1) {$ab\rightarrow bg$};
    \end{scope}

    \begin{scope}[scale=.8, xshift=200, yshift=0]
	\vertex[fill, label=left:\footnotesize{$a$}](v1) at (0,0) {};
	\vertex[fill, label=above:\footnotesize{$b$}](v2) at (0,1.5) {};
	\vertex[fill, label=above:\footnotesize{$c$}](v3) at (1.25,.75) {};
	\vertex[fill, label=below:\footnotesize{$d$}](v4) at (1.25,-.75) {};
	\vertex[fill, label=below:\footnotesize{$e$}](v5) at (0,-1.5) {};	
	\vertex[fill, label=below:\footnotesize{$f$}](v6) at (-1.25,-.75) {};
	\vertex[fill, label=above:\footnotesize{$g$}](v7) at (-1.25,.75) {};

	\draw[thick] (v7)--(v2);
	\draw[thick] (v1)--(v3);
	\draw[thick] (v1)--(v4);
	\draw[thick] (v1)--(v5);	
	\draw[thick] (v1)--(v6);
	\draw[thick] (v1)--(v7);
	\draw[thick] (v2)--(v3);
	\draw[thick] (v4)--(v5);
  
    \end{scope}

    \begin{scope}[scale=.8,xshift=260, yshift=0]
     \draw[thick, -stealth] (0,0) to[out=30,in=150] (3,0);  
     \node[] at (1.5,1) {$ad\rightarrow cd$};
    \end{scope}

    \begin{scope}[scale=.8, xshift=400, yshift=0]
	\vertex[fill, label=left:\footnotesize{$a$}](v1) at (0,0) {};
	\vertex[fill, label=above:\footnotesize{$b$}](v2) at (0,1.5) {};
	\vertex[fill, label=above:\footnotesize{$c$}](v3) at (1.25,.75) {};
	\vertex[fill, label=below:\footnotesize{$d$}](v4) at (1.25,-.75) {};
	\vertex[fill, label=below:\footnotesize{$e$}](v5) at (0,-1.5) {};	
	\vertex[fill, label=below:\footnotesize{$f$}](v6) at (-1.25,-.75) {};
	\vertex[fill, label=above:\footnotesize{$g$}](v7) at (-1.25,.75) {};

	\draw[thick] (v7)--(v2);
	\draw[thick] (v1)--(v3);
	\draw[thick] (v3)--(v4);
	\draw[thick] (v1)--(v5);	
	\draw[thick] (v1)--(v6);
	\draw[thick] (v1)--(v7);
	\draw[thick] (v2)--(v3);
	\draw[thick] (v4)--(v5);
  
    \end{scope}
    
    \end{tikzpicture}

    \caption{An example of a walk through \(\mathcal{G}(7,8)\) demonstrating a sequence of two possible single-edge replacements, first replacing \(ab\) with \(bg\), then replacing \(ad\) with \(cd\).}
    \label{fig:single edge replacement}
\end{figure}

Thus, note that every arrow in \(\G(n,m)\) is reversible.
Further, for an edge \(e\in G\) and \(f\) in the complement of \(G\), if the edge set \((E(G)\setminus\{e\})\cup\{f\}\) does not form a connected graph, define \(\G(n,m)\) to have a loop at \(G\).
It is straightforward to show that the directed graph \(\G(n,m)\) is regular, strongly connected, and aperiodic, hence we can conclude that ensembles generated by this method asymptotically obey a uniform distribution~\cite{SIAMMCMC}.

Our sampling method begins by generating a random element \(G\) of \(\G(n,m)\) and then successively randomly choosing an edge \(e\in E(G)\) and a non-edge \(f\in E(G)^C\) to generate the next step in a random walk on the graph of graphs.
We use subsampling, typically taking every 11-th graph, in an attempt to generate an ensemble of graphs with more diverse structures, though, due to computational constraints, our sample sizes are not particularly large.
Note that we are not selecting our subsampling frequency based on any information about the target distribution or mixing time.

\subsubsection{Fixed Degree Sequence}

To sample from the space of simple connected graphs on \(n\) vertices with a fixed degree sequence, we employ a double-edge swap MCMC technique as described in~\cite[Section~2]{SIAMMCMC}.
Define the graph of graphs \(\mathcal{G}(\mathbf{d})\) to be the directed graph with vertex set all connected graphs with degree sequence \(\mathbf{d}\).
A connected graph \(G\) has an arrow to \(G'\) in \(\mathcal{G}(\mathbf{d})\) if \(G'\) is obtained from \(G\) via a double-edge swap, i.e., if there exist edges \(uv\) and \(xy\) in \(G\) such that replacing these edges with \(ux\) and \(vy\) produces \(G'\). 
An example is shown in Figure~\ref{fig:double edge swap}.

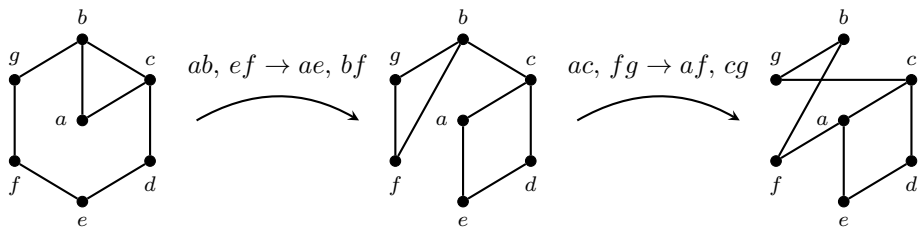
\begin{figure}[h!]
    \centering

 \begin{tikzpicture}[scale=0.9]
    
    \begin{scope}[scale=.8, xshift=0, yshift=0]
	\vertex[fill, label=left:\footnotesize{$a$}](v1) at (0,0) {};
	\vertex[fill, label=above:\footnotesize{$b$}](v2) at (0,1.5) {};
	\vertex[fill, label=above:\footnotesize{$c$}](v3) at (1.25,.75) {};
	\vertex[fill, label=below:\footnotesize{$d$}](v4) at (1.25,-.75) {};
	\vertex[fill, label=below:\footnotesize{$e$}](v5) at (0,-1.5) {};	
	\vertex[fill, label=below:\footnotesize{$f$}](v6) at (-1.25,-.75) {};
	\vertex[fill, label=above:\footnotesize{$g$}](v7) at (-1.25,.75) {};
	
	\draw[thick] (v1)--(v2);
	\draw[thick] (v1)--(v3);
	\draw[thick] (v2)--(v7);
    \draw[thick] (v2)--(v3);
	\draw[thick] (v3)--(v4);	
	\draw[thick] (v4)--(v5);
	\draw[thick] (v5)--(v6);
    \draw[thick] (v6)--(v7);
  
    \end{scope}

    \begin{scope}[scale=.8,xshift=60, yshift=0]
     \draw[thick, -stealth] (0,0) to[out=30,in=150] (3,0);
          \node[] at (1.5,1) {$ab,\,ef\rightarrow ae,\, bf$};
    \end{scope}

    \begin{scope}[scale=.8, xshift=200, yshift=0]
	\vertex[fill, label=left:\footnotesize{$a$}](v1) at (0,0) {};
	\vertex[fill, label=above:\footnotesize{$b$}](v2) at (0,1.5) {};
	\vertex[fill, label=above:\footnotesize{$c$}](v3) at (1.25,.75) {};
	\vertex[fill, label=below:\footnotesize{$d$}](v4) at (1.25,-.75) {};
	\vertex[fill, label=below:\footnotesize{$e$}](v5) at (0,-1.5) {};	
	\vertex[fill, label=below:\footnotesize{$f$}](v6) at (-1.25,-.75) {};
	\vertex[fill, label=above:\footnotesize{$g$}](v7) at (-1.25,.75) {};

	\draw[thick] (v1)--(v5);
	\draw[thick] (v1)--(v3);
	\draw[thick] (v2)--(v7);
    \draw[thick] (v2)--(v3);
	\draw[thick] (v3)--(v4);	
	\draw[thick] (v4)--(v5);
	\draw[thick] (v2)--(v6);
    \draw[thick] (v6)--(v7);
  
    \end{scope}

    \begin{scope}[scale=.8,xshift=260, yshift=0]
     \draw[thick, -stealth] (0,0) to[out=30,in=150] (3,0);
               \node[] at (1.5,1) {$ac,\,fg\rightarrow af,\, cg$};
    \end{scope}

    \begin{scope}[scale=.8, xshift=400, yshift=0]
	\vertex[fill, label=left:\footnotesize{$a$}](v1) at (0,0) {};
	\vertex[fill, label=above:\footnotesize{$b$}](v2) at (0,1.5) {};
	\vertex[fill, label=above:\footnotesize{$c$}](v3) at (1.25,.75) {};
	\vertex[fill, label=below:\footnotesize{$d$}](v4) at (1.25,-.75) {};
	\vertex[fill, label=below:\footnotesize{$e$}](v5) at (0,-1.5) {};	
	\vertex[fill, label=below:\footnotesize{$f$}](v6) at (-1.25,-.75) {};
	\vertex[fill, label=above:\footnotesize{$g$}](v7) at (-1.25,.75) {};

	\draw[thick] (v1)--(v5);
	\draw[thick] (v1)--(v3);
	\draw[thick] (v2)--(v7);
    \draw[thick] (v7)--(v3);
	\draw[thick] (v3)--(v4);	
	\draw[thick] (v4)--(v5);
	\draw[thick] (v2)--(v6);
    \draw[thick] (v6)--(v1);
  
    \end{scope}
    
    \end{tikzpicture}

    \caption{An example of a walk through \(\mathcal{G}(\{3,3,2,2,2,2,2\})\) demonstrating a sequence of two possible double edge swaps, first swapping the endpoints of \(ab\) and \(ef\), then swapping the endpoints of \(ac\) and \(fg\).}
    \label{fig:double edge swap}
\end{figure}

If performing a particular double-edge swap on \(G\) would produce a graph that is outside the space (i.e. the new graph has a loop or multiedge or is disconnected), that swap will correspond to a loop on the vertex \(G\) in \(\mathcal{G}(\mathbf{d})\). 
It is shown in~\cite{SIAMMCMC} that \(\mathcal{G}(\mathbf{d})\) is regular, strongly connected, and aperiodic.
Thus, as before, the samples asymptotically obey a uniform distribution.

Our sampling method begins by generating a connected graph with degree sequence \(\mathbf{d}\) via the Havel-Hakimi algorithm~\cite{Havel-Hakimi} and randomly performing double-edge swaps.
Again, we employ subsampling, typically taking every 5-th or 11-th graph depending on the number of vertices.
As before, our subsampling frequency is not based on any information regarding the target distribution or mixing time.

\subsection{Data and Observations}\label{subsec:vert_edge_data}

In each of our experiments, we generated an ensemble of graphs with specified invariants: number of vertices and either number of edges or degree sequence.
For each of our samples, we computed the average local clustering and the number of facets for $\PG$, and we generated a plot displaying the results.
We discuss these experiments and results in this subsection.

\subsubsection{Fixed Number of Edges}

We first used single-edge swap MCMC methods to generate ensembles of graphs with a fixed number of vertices and edges.
We computed \(\Cws(G)\) and \(\facets(\PG)\) for each graph in our ensemble, and plotted the resulting ordered pairs.
In each of these plots, the number of facets appears to generally increase as \(\Cws\) increases.
Additionally, we observe that these plots often exhibit heteroscedasticity, i.e., the variance of the data changes as \(\Cws\) increases. 
Figure~\ref{fig: 11vert} shows two representatives of the types of plots we observe.

\begin{figure}[h]
    \centering
    \subfloat[1001 graphs with 11 vertices and 25 edges.]{\includegraphics[width=0.49\textwidth]{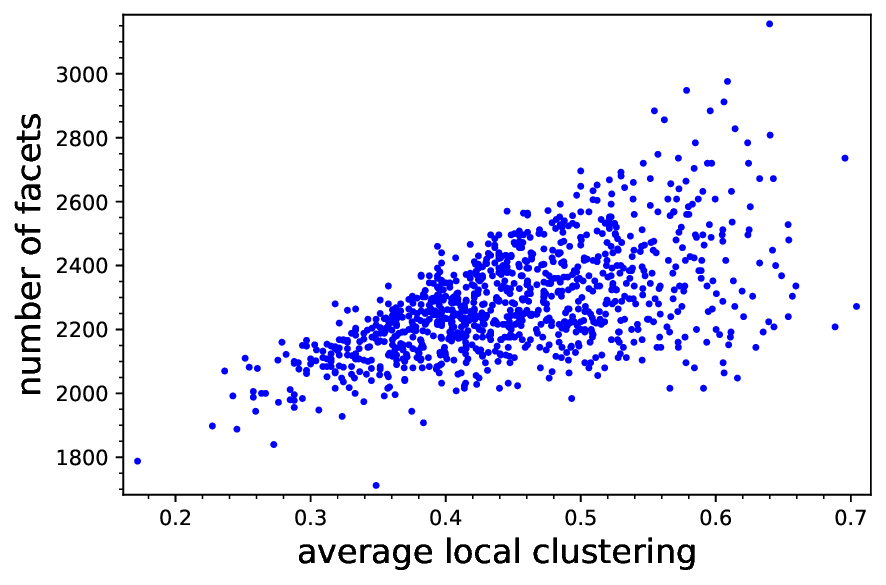}}
    \subfloat[201 connected graphs with 15 vertices and 37 edges]{\includegraphics[width=0.49\textwidth]{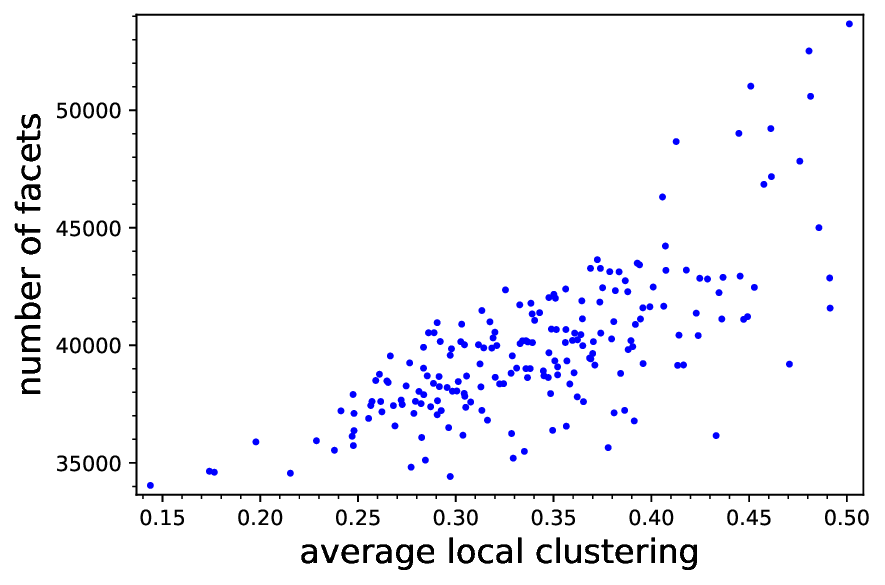}}
    \caption{Fixed edge data for 11 and 15 vertices}
    \label{fig: 11vert}
\end{figure}

Figure~\ref{fig:table of single edge plots} contains plots for connected graphs on $11$ vertices with various fixed numbers of edges.
These plots suggest that the heteroscedasticity phenomenon, where the variance in the number of facets increases as average local clustering increases, arises across multiple fixed edge counts.
It is important to note that in the data plots for Figure~\ref{fig:table of single edge plots}, all the axes change scale.
Thus, for example, the plot for $11$ vertices and $20$ edges has average local clustering range from near $0$ to $0.8$, and facet numbers range from under $2000$ to over $4000$.
However, the plot with $35$ edges has a significantly restricted range for both the horizontal and vertical axis.
This is the same phenomenon that appeared in Figure~\ref{fig:table of plots}, where higher edge density yields less variation for both average local clustering and facet numbers.
Nevertheless, even at different scales, a positive correlation is observed.

\begin{figure}[h]
\subfloat[11 vertices, 20 edges]{\includegraphics[scale=0.49]{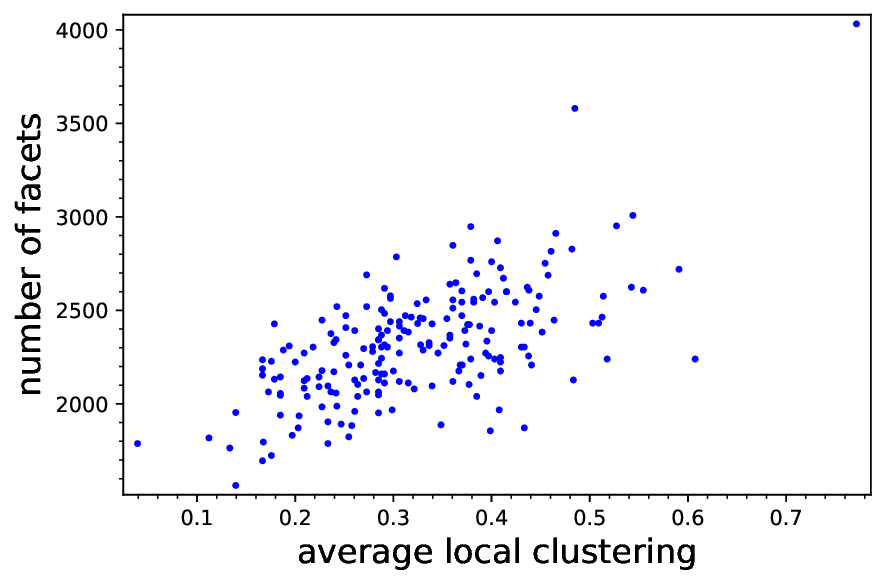}}
\subfloat[11 vertices, 25 edges]{\includegraphics[scale=0.49]{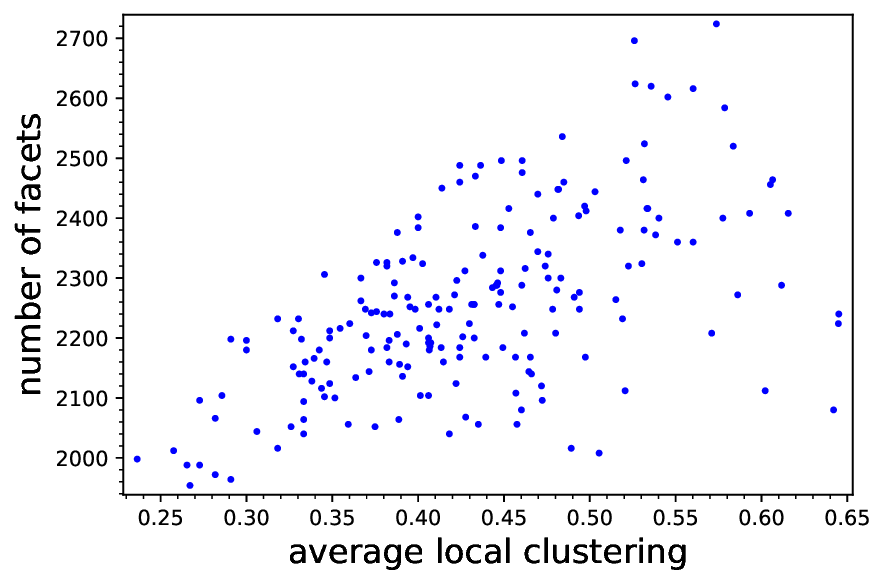}}\\
\subfloat[11 vertices, 30 edges]{\includegraphics[scale=0.49]{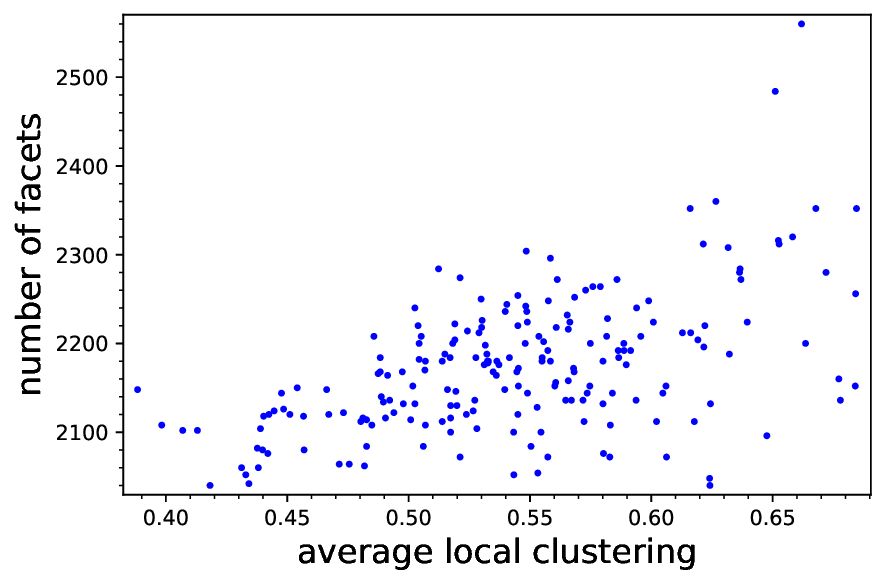}}
\subfloat[11 vertices, 35 edges]{\includegraphics[scale=0.49]{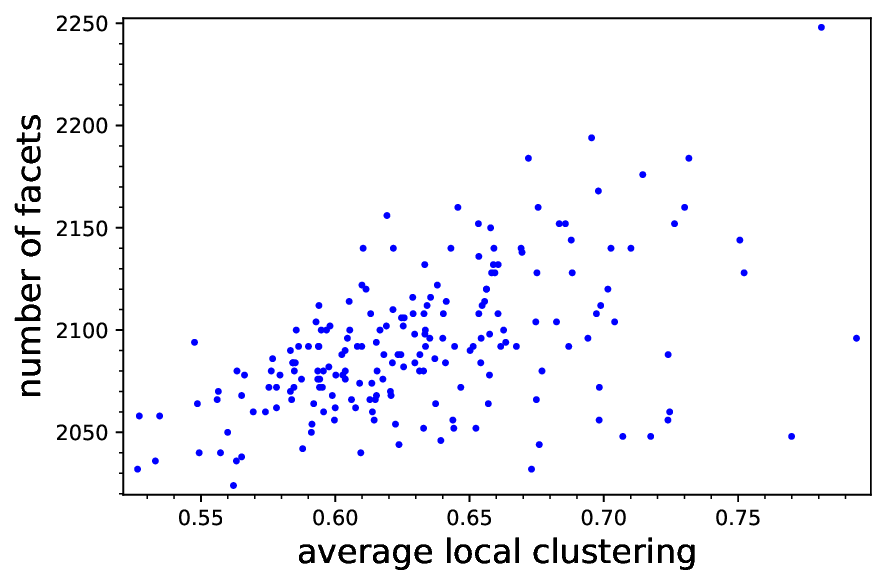}}
    \caption{Data for graphs on 11 vertices with varying edge numbers}
    \label{fig:table of single edge plots}
\end{figure}

\subsubsection{Graphs with Hubs}\label{sec: hubs}

Our next experiments used double-edge swap MCMC methods to generate connected graphs with a fixed degree sequence.
In real-world graphs, it is common for there to be a large number of lower-degree vertices and a small number of higher-degree vertices; the latter are often referred to as \emph{hubs}.
This has led to the development of various random graph models that exhibit scale-free degree distributions~\cite{newmannetworks}.
Because we are limited in the dimensions of \(\PG\) for which we can effectively compute the number of facets, the magnitude of hubs that we can study are not as great as often found in large real-world networks.
However, Figures~\ref{fig:18verthub} and~\ref{fig:17verthub} are representative of the data we have observed in ensembles of graphs on less than $20$ vertices where the degree sequence has a small number of high-degree vertices.
In all of the experiments we have conducted for graphs with hubs, a  correlation between \(\Cws\) and \(\facets(G)\) is observed.

\begin{figure}[h]
\centering
\begin{minipage}{0.49\textwidth}
        \centering
        \includegraphics[width=\textwidth]{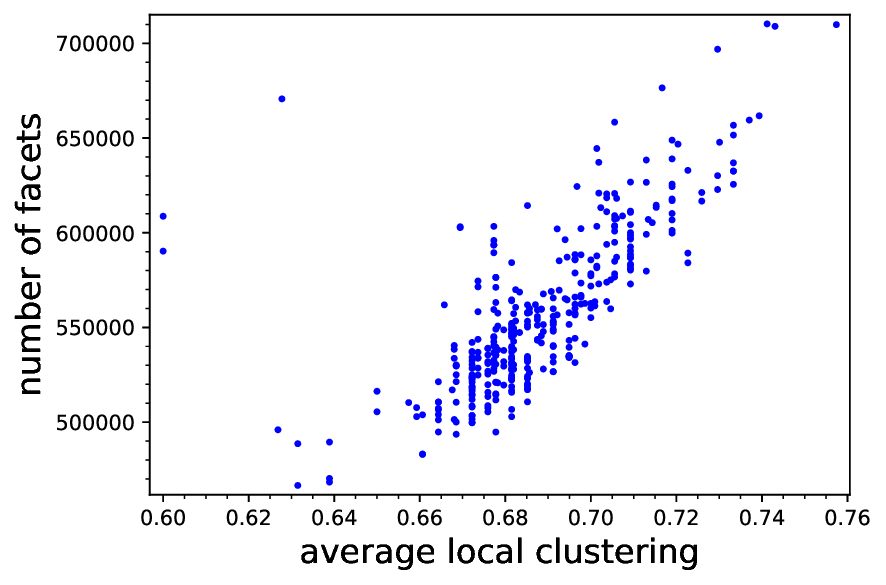}
        \caption{Data from 370 connected graphs having 18 vertices and degree sequence \([3,3,4,4,\ldots,4,4,5,5,16,16]\).}
    \label{fig:18verthub}
    \end{minipage}
    \hfill
    \begin{minipage}{0.49\textwidth}
        \centering
        \includegraphics[width=\textwidth]{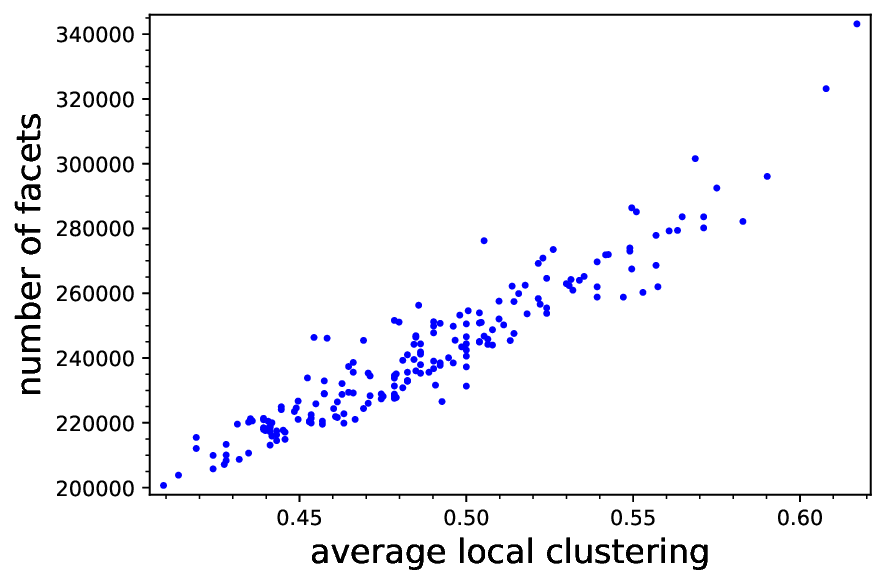}    \caption{Data from an ensemble of 192 connected graphs with 17 vertices and degree sequence \([3,3,3,4,4,\ldots,4,4,5,5,5,5,15]\) obtained by MCMC with double-edge swaps.}
        \label{fig:17verthub}
    \end{minipage}
\end{figure}

\subsubsection{\(k\)-Regular Graphs}\label{sec: regular}

A classic family of graphs with a fixed degree sequence are \(k\)-regular graphs, i.e., graphs where every vertex has degree \(k\).
The number of connected regular graphs on \(n\) vertices is a well-studied integer sequence~\cite[Sequence A005177]{OEIS}.
Based on our previous observations, for larger values of \(k\) the average local clustering should be higher, due to higher edge density.
What is less clear is what to expect from the number of facets of the symmetric edge polytope as \(k\) varies, especially for small \(k\).

Figure~\ref{fig:12vertregularsall} provides a plot of average local clustering and number of facets for 3390 connected regular graphs on 12 vertices (that there are 18979 such graphs), sampled using double-edge MCMC for each \(k\).
Note that for small values of \(k\), there is a trend that \(\facets(G)\) increases as \(\Cws\) increases.
When \(k\) is small, the average local clustering is generally less than \(0.4\) and the number of facets varies widely.
As \(k\) increases, the average local clustering varies less, and the number of facets concentrates near the value of \(\facets(K_{12})=4094\).
In general, it is reasonable to expect that as the edge density of a \(k\)-regular graph \(G\) increases, and thus as the graph becomes closer to a complete graph, there will be many connected spanning bipartite subgraphs where the induced subgraph of \(G\) on each shore of the bipartition is connected.

\begin{figure}[h]
    \centering
        \includegraphics[width=0.5\textwidth]{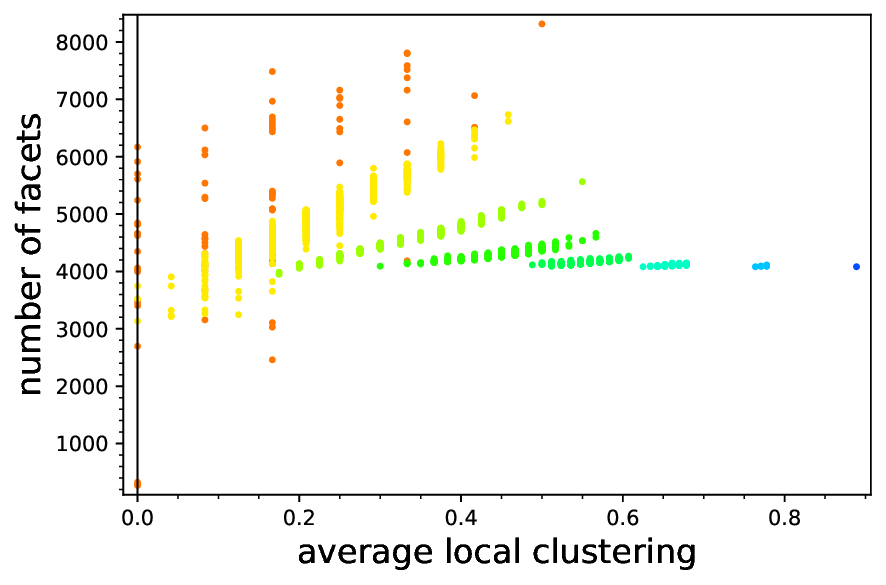} 
        \caption{Data from a sample of 3390 connected \(k\)-regular graphs on 12 vertices obtained by MCMC with double-edge swaps, for \(k=3,4,5,6,7,8,9,10\).
        Each value of \(k\) corresponds to a different color in the plot, with lower \(k\) having smaller \(\Cws\) values.
        Note that for larger \(k\), the number of facets is approximately \(\facets(K_{12})\).
        }
    
    \label{fig:12vertregularsall}
\end{figure}

Additional data plots from ensembles of \(k\)-regular graphs on 18 vertices for \(k=3,7\) are given in Figures~\ref{fig:18vert3red} and~\ref{fig:18vert7reg}.
Both of these plots further illustrate the phenomenon shown in Figure~\ref{fig:12vertregularsall}, where \(k\)-regular graphs for smaller \(k\) have significantly larger variance in the number of facets (as seen in the range of the vertical axes), and have lower average local clustering.

\begin{figure}[h]
    \centering
    \begin{minipage}{0.49\textwidth}
        \centering
        \includegraphics[width=\textwidth]{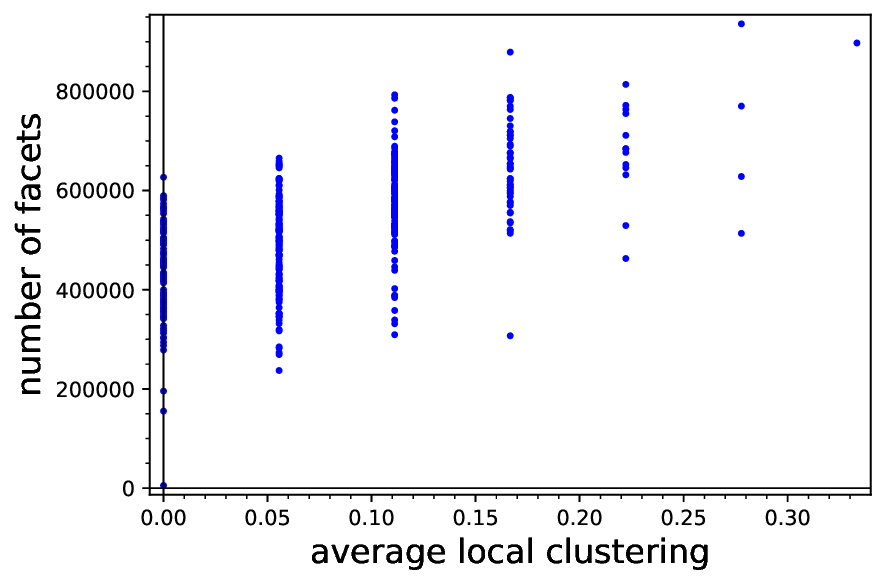} 
        \caption{Data from an ensemble of 397 connected \(3\)-regular graphs on 18 vertices obtained by MCMC using double-edge swaps.
        Note the large range of the vertical axis.}
        \label{fig:18vert3red}
    \end{minipage}
    \hfill
    \begin{minipage}{0.49\textwidth}
        \centering
        \includegraphics[width=\textwidth]{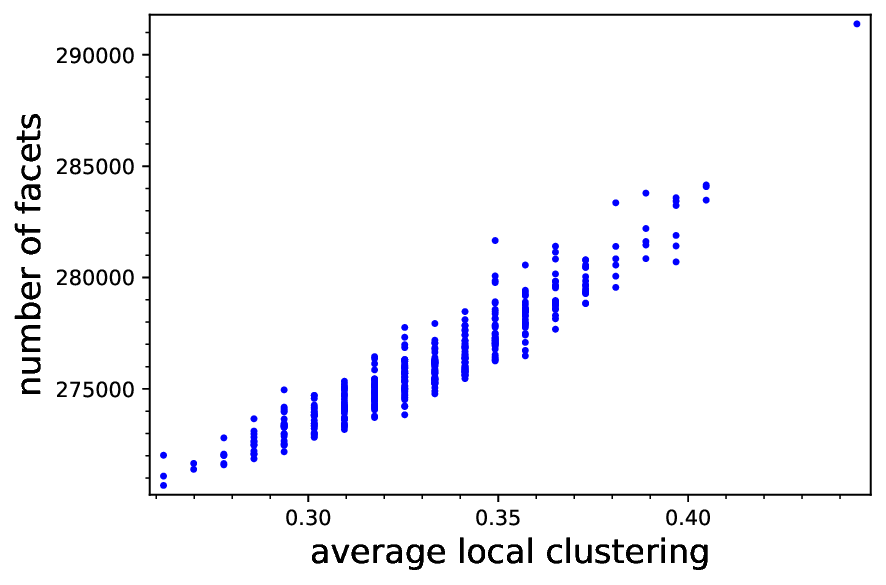}    \caption{Data from an ensemble of 399 connected \(7\)-regular graphs on 18 vertices obtained by MCMC using double-edge swaps.
        Note the narrow range on the vertical axis.}
        \label{fig:18vert7reg}
    \end{minipage}
\end{figure}

\section{Discussion}\label{sec:discussion}

It is extremely challenging to understand the facet structure of \(\PG\) for a random graph generated by any of the models we have considered in this work.
We will discuss an example of a toy theoretical result in this direction prior to our final discussion.

\subsection{Case: \(n\) Vertices, \(n\) Edges, Fixed Degree Sequence}
When the number of vertices is equal to the number of edges, we can compute \(\facets(\PG)\) for any connected graph with a fixed degree sequence. 
Further, we can describe a graph with that degree sequence that attains the maximum number of facets.
This is possible because, in this case, \(\facets(\PG)\) depends only on the length of the unique cycle in \(G\).
To get this facet information for a given degree sequence, we need only know what lengths of cycle are attainable with that sequence.

\begin{prop}\label{prop: cycle lengths by degree sequence}
Let \(G\) be a simple graph with \(n\) vertices and \(n\) edges with degree sequence \(\mathbf{d}=\{d_i\}_{i=1}^n\) (\(d_i\geq d_{i+1}\)). Let \(m_G\) denote the minimum possible length of a cycle in \(G\), and let \(M_G\) denote the maximum possible length of a cycle in \(G\).
\begin{enumerate}
    \item[(i)] If \(d_i=2\) for all \(i\), \(m_G=M_G=n\).
    \item[(ii)] If \(d_k\geq 2\) and \(d_i=1\) for \(i>k\) with \(k<n\), \(m_G=3\) and \(M_G=k\).
\end{enumerate}
\end{prop}

\begin{proof}
To show (i), note that the only simple connected graph with this degree sequence is the \(n\) cycle.  
So \(m_G=M_G=n\).

To show (ii), we construct a simple, connected graph \(G\) with a cycle of length \(m\) for \(3\leq m\leq k\) with the following procedure. 
\begin{enumerate}
    \item Construct a cycle, \(C\), on vertices \(1,\dots,m\).
    \item If \(m=k\), skip to (4). Otherwise, construct a path, \(P\), on vertices \(m+1,\dots, k\) with \(k-(m+1)\geq 0\) edges. 
    \item Now, \(d_1\geq 3\) and \(d_{m+1}\geq 2\), so vertex 1 is incident to at least one edge not on \(C\) and vertex \(m+1\) is incident to at least one edge not on \(P\). So we connect 1 and \(m+1\).
    \item The remaining edges are incident to leaves. There are \(n-k\) leaves and 
    \begin{equation*}
        \begin{split}
            &n-(\text{number of edges of \(C\)})-(\text{number of edges of \(P\)})-1\\&= n-m-(k-m-1)-1 \\&= n-k
        \end{split}
    \end{equation*}
    edges that must be added. So there are exactly enough open half-edges on the vertices \(1,\dots, k\) to be filled by the leaves.
\end{enumerate}

\end{proof}

\begin{figure}[h]
\begin{center}

\begin{tikzpicture}
\begin{scope}[scale=1, xshift=0, yshift=0]
	\vertex[fill,label=below:\footnotesize{$1$}](a1) at (1,0) {};
	\vertex[fill,label=below:\footnotesize{$2$}](a2) at (2,0) {};
	\vertex[fill,label=below:\footnotesize{$3$}](a3) at (3,0) {};
	\vertex[fill,label=below:\footnotesize{$4$}](a4) at (4,0) {};
	\vertex[fill,label=below:\footnotesize{$5$}](a5) at (5,0) {};
	\vertex[fill,label=below:\footnotesize{$6$}](a6) at (6,0) {};

    \node[] at (1,0.3) {\textcolor{cyan}{\tiny$3$}};
    
    \node[] at (2, 0.3) {\textcolor{cyan}{\tiny$3$}};
    
    \node[] at (3,0.3) {\textcolor{cyan}{\tiny$2$}};
    
    \node[] at (4,0.3) {\textcolor{cyan}{\tiny$2$}};

    \node[] at (5,0.3) {\textcolor{cyan}{\tiny$1$}};

    \node[] at (6,0.3) {\textcolor{cyan}{\tiny$1$}};

\end{scope}
\end{tikzpicture}

\begin{tikzpicture}
\begin{scope}[scale=1, xshift=0, yshift=0]
	\vertex[fill,label=below:\footnotesize{$1$}](a1) at (1,0) {};
	\vertex[fill,label=below:\footnotesize{$2$}](a2) at (2,0) {};
	\vertex[fill,label=below:\footnotesize{$3$}](a3) at (3,0) {};
	\vertex[fill,label=below:\footnotesize{$4$}](a4) at (4,0) {};
	\vertex[fill,label=below:\footnotesize{$5$}](a5) at (5,0) {};
	\vertex[fill,label=below:\footnotesize{$6$}](a6) at (6,0) {};
	
	\draw[thick] (a1)--(a2);
	\draw[thick] (a2)--(a3);
	\draw[thick] (a1) to[out=60,in=120] (a3);

    \node[] at (1,0.3) {\textcolor{cyan}{\tiny$3$}};
    
    \node[] at (2, 0.3) {\textcolor{cyan}{\tiny$3$}};
    
    \node[] at (3,0.3) {\textcolor{cyan}{\tiny$2$}};
    
    \node[] at (4,0.3) {\textcolor{cyan}{\tiny$2$}};

    \node[] at (5,0.3) {\textcolor{cyan}{\tiny$1$}};

    \node[] at (6,0.3) {\textcolor{cyan}{\tiny$1$}};

\end{scope}
\end{tikzpicture}

\begin{tikzpicture}
\begin{scope}[scale=1, xshift=0, yshift=0]
	\vertex[fill,label=below:\footnotesize{$1$}](a1) at (1,0) {};
	\vertex[fill,label=below:\footnotesize{$2$}](a2) at (2,0) {};
	\vertex[fill,label=below:\footnotesize{$3$}](a3) at (3,0) {};
	\vertex[fill,label=below:\footnotesize{$4$}](a4) at (4,0) {};
	\vertex[fill,label=below:\footnotesize{$5$}](a5) at (5,0) {};
	\vertex[fill,label=below:\footnotesize{$6$}](a6) at (6,0) {};
	
	\draw[thick] (a1)--(a2);
	\draw[thick] (a2)--(a3);
	\draw[thick] (a1) to[out=60,in=120] (a3);
	\draw[thick] (a1) to[out=60,in=120] (a4);

    \node[] at (1,0.3) {\textcolor{cyan}{\tiny$3$}};
    
    \node[] at (2, 0.3) {\textcolor{cyan}{\tiny$3$}};
    
    \node[] at (3,0.3) {\textcolor{cyan}{\tiny$2$}};
    
    \node[] at (4,0.3) {\textcolor{cyan}{\tiny$2$}};

    \node[] at (5,0.3) {\textcolor{cyan}{\tiny$1$}};

    \node[] at (6,0.3) {\textcolor{cyan}{\tiny$1$}};

\end{scope}
\end{tikzpicture}

\begin{tikzpicture}
\begin{scope}[scale=1, xshift=0, yshift=0]
	\vertex[fill,label=below:\footnotesize{$1$}](a1) at (1,0) {};
	\vertex[fill,label=below:\footnotesize{$2$}](a2) at (2,0) {};
	\vertex[fill,label=below:\footnotesize{$3$}](a3) at (3,0) {};
	\vertex[fill,label=below:\footnotesize{$4$}](a4) at (4,0) {};
	\vertex[fill,label=below:\footnotesize{$5$}](a5) at (5,0) {};
	\vertex[fill,label=below:\footnotesize{$6$}](a6) at (6,0) {};
	
	\draw[thick] (a1)--(a2);
	\draw[thick] (a2)--(a3);
	\draw[thick] (a1) to[out=60,in=120] (a3);
	\draw[thick] (a1) to[out=60,in=120] (a4);
	\draw[thick] (a4) to[out=60,in=120] (a6);
	\draw[thick] (a2) to[out=-60,in=-120] (a5);

    \node[] at (1,0.3) {\textcolor{cyan}{\tiny$3$}};
    
    \node[] at (2, 0.3) {\textcolor{cyan}{\tiny$3$}};
    
    \node[] at (3,0.3) {\textcolor{cyan}{\tiny$2$}};
    
    \node[] at (4,0.3) {\textcolor{cyan}{\tiny$2$}};

    \node[] at (5,0.3) {\textcolor{cyan}{\tiny$1$}};

    \node[] at (6,0.3) {\textcolor{cyan}{\tiny$1$}};

\end{scope}
\end{tikzpicture}
\end{center}
\caption{Construction of a graph with degree sequence \(\{3,3,2,2,1,1\}\) containing a 3-cycle via the algorithm in Proposition~\ref{prop: cycle lengths by degree sequence}.  The degree of each vertex is given in blue above the vertex.}
\label{fig:prop 3.2 construction}
\end{figure}
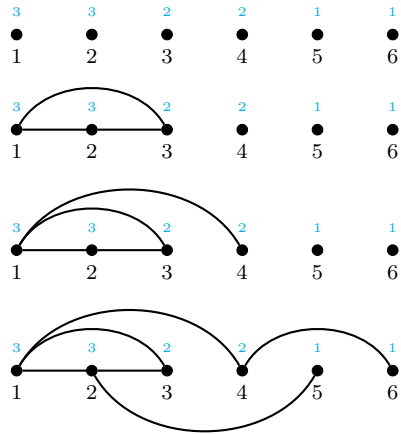

\begin{example}
Consider the degree sequence \(\{3,3,2,2,1,1\}\) for a graph on \(6\) vertices with \(6\) edges. Following the steps in Proposition~\ref{prop: cycle lengths by degree sequence}, we construct a graph with this degree sequence containing a 3-cycle (\(m=3\)) as follows. An illustration of this construction is given in Figure~\ref{fig:prop 3.2 construction}.
\begin{enumerate}
    \item Construct the 3-cycle \(C\) on the vertices \(1,2,3\) each of which have degree at least 2.
    The vertex \(3\) now has the desired degree.
    \item Since \(m=3<4=k\), we construct the path \(P\) containing only the vertex \(4\) (a path with one vertex and no edges).
    \item Since the degree of vertex \(1\) is \(3\), and it is incident to only two edges on \(C\), we can add an edge between vertices \(1\) and \(4\) to connect the cycle to the path.
    The vertex \(1\) now has the desired degree.
    \item Now, we must add one edge incident to each of the vertices \(2\) and \(4\).
    We have exactly enough leaves to add these necessary edges.
    Adding the edges \(\{2,5\}\) and \(\{4,6\}\) gives the desired degree for vertices \(2\), \(4\), \(5\), and \(6\).
\end{enumerate}
\end{example}

Given Proposition~\ref{prop: cycle lengths by degree sequence}, we can apply \cite[Theorem~3.2]{braunbruegge2022facets} to a class of graphs with \(n\) vertices and \(n\) edges that have a specific degree sequence to identify the facet-maximizing graphs for that degree sequence.
As in~\cite{braunbruegge2022facets}, we use the notation \(G\vee H\) to denote a graph obtained by identifying graphs \(G\) and \(H\) at a single vertex. 

\begin{corollary}\label{cor:nn bound}
Let \(G\) be a simple graph with \(n\) vertices and \(n\) edges with degree sequence \(\{d_i\}_{i=1}^n\) (\(d_i\geq d_{i+1}\)). Let \(\facets(G)\) denote the number of facets of the symmetric edge polytope \(\PG\).
\begin{enumerate}
    \item[(i)] If \(d_n=2\), 
    \[
    \facets(G)=\facets(C_n).
    \]
    \item[(ii)] If \(d_k>1\) and \(d_{k+1}=1\) for some \(k<n\), and \(\ell\) is the largest odd number satisfying \(\ell\leq k\),
    \[
    \facets(G)\leq \facets(C_\ell \vee P_{n-\ell})
    \]
\end{enumerate}
Here, \(C_m\) denotes a cycle with \(m\) edges, and \(P_m\) denotes a path with \(m\) edges.
\end{corollary}

Even further than this, the proof of~\cite[Theorem 3.2]{braunbruegge2022facets} describes how \(\facets(G)\) changes as the cycle length varies.
For any degree sequence that allows a cycle of length at least \(5\), a graph \(G\) that maximizes \(\facets(G)\) has \(\Cws(G) =0\). 
In fact, the only graphs on \(n\) vertices and \(n\) edges that have nonzero average local clustering are those that have a single 3-cycle. 
Even these graphs have \(\Cws\) approaching 0 as \(n\) increases.

A precise theoretical result was attainable in this case because the small number of edges relative to number of vertices significantly restricts the structure of the graphs we consider.

\subsection{Final Discussion}
The goal of our current investigation of symmetric edge polytopes was to study relationships between graph structure and facet structure, with a focus on both theoretical and empirical results for various random graph models.
For Erd\"os-Renyi random graphs \(G~\sim~G(n,p)\), we observed empirically that  as $p$ increases, the number of facets of \(\PG\) tends toward \(2^n-2\), the number of facets of \(P_{K_n}\).
We also established a threshold of \(p>1/2\) such that, with high probability, any sufficiently even bipartition \((A, V\setminus A)\) of the vertices of \(G\sim G(n,p)\) induces a facet subgraph of \(G\). 
Thus, with high probability, the number of such bipartitions gives a lower bound on \(\facets(\PG)\) for Erd\"os-Renyi graphs.
Furthermore, with high probability, \(G\) has a facet subgraph that supports exactly the facet hyperplanes supported by the same bipartition on the complete graph. 

While the trends seen when we sample from \(G(n,p)\) have no apparent connection to the average local clustering coefficient, which is \(p\), correlations are observed when we consider different approaches to sampling graphs. 
An exhaustive computation for all connected graphs on eight vertices shows a positive correlation between average local clustering and number of facets.
Using Markov Chain Monte Carlo methods to sample from the space of connected graphs with a fixed number of vertices and edges, we see that graphs with higher \(\Cws\) tend to produce polytopes with more facets. 
Additionally, data from these ensembles indicates that the range in facet counts tends to increase as \(\Cws\) increases. 
When we further restrict to the space of connected graphs with fixed degree sequence (in particular graphs with hubs and \(k\)-regular graphs), this correlation remains and we observe less variety in the range of facet counts across difference small intervals of clustering values. 

\begin{figure}[h]
    \centering
        \includegraphics[width=0.6\textwidth]{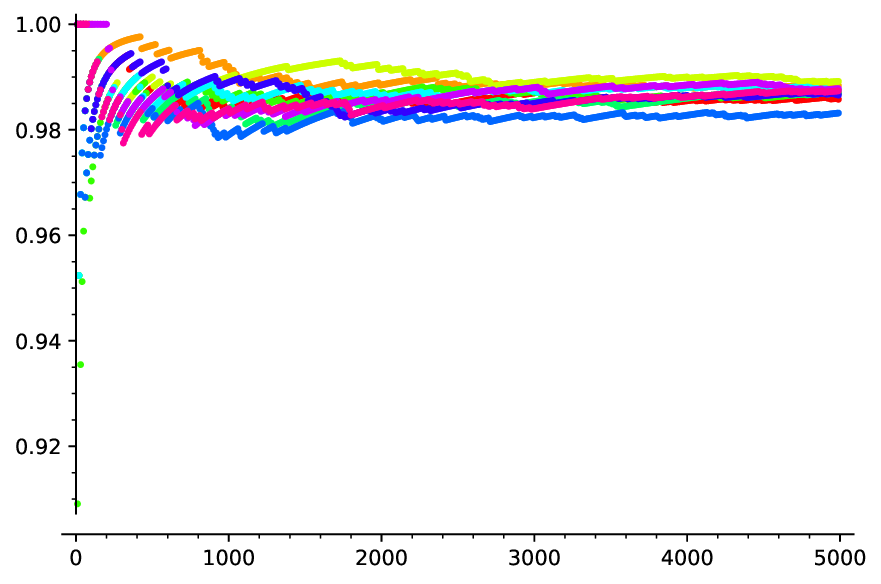} 
        \caption{A collection of sequence plots for a sample of ten 11-regular connected graphs \(G\) on 5000 vertices showing how the fraction of sampled subsets \(A_i\) inducing a connected \(\bis{A_i}\) changes over time and appears to stabilize near a value between \(0.98\) and \(1\).}
    
    \label{fig:Ef11regular5000}
\end{figure}

In service of counting or bounding the number of facets of symmetric edge polytopes, it would be useful to understand how many and which induced bipartite subgraphs are facet subgraphs for a given \(G\), particularly in the case where \(G\) is connected and sparse. 
Toward this end, we can consider plots such as Figure~\ref{fig:Ef11regular5000}, which shows the results of an experiment identifying facet subgraphs of some 11-regular graphs on 5000 vertices.
To create this plot, we generated ten 11-regular connected graphs on 5000 vertices using double-edge swap MCMC, sampling after every 100001 swaps. 
For each graph \(G\), a sequence of 5000 random subsets \((A_1,\ldots,A_{5000})\) of the vertex set \(V\) was generated. 
For each \(0<10j\leq 5000\), we compute the fraction \(b_{10j}\) of the subsets in \((A_1,\ldots,A_{10j})\) which induce connected bipartite subgraphs \(\bis{A_i}\) and plot the point \((10j,b_{10j})\). 
This process yields a sequence plot for each sampled graph. 
As can be seen in this figure, as \(j\) increases, for each graph the fraction of subsets inducing a connected bipartite subgraph appears to stabilize near a value between \(0.98\) and \(1\). 
Note that this is an extremely small sample of the \(2^{5000}\) subsets of the vertex set, and a small sample of \(11\)-regular graphs. 
Nonetheless, these results are surprising given that these graphs, though large, are sparse with only \(27,500\) of \(12,497,500\) possible edges, or approximately \(0.22\%\). 
An open question of interest is to determine, for a fixed \(k\), asymptotic estimates for the expected number of induced bipartite subgraphs that are facet subgraphs for a \(k\)-regular graph on \(n\) vertices.

\acknowledgements
\label{sec:ack}
MK was partially supported by National Science Foundation award DMS-2005630. BB and KB were partially supported by National Science Foundation award DMS-1953785.
The authors thank Tianran Chen and Rob Davis for helpful discussions that motivated this project.
The authors thank Dhruv Mubayi for helpful suggestions regarding random graphs.

\nocite{*}
\bibliographystyle{alpha}

\begin{thebibliography}{10}

\bibitem{OEIS}
The On-Line Encyclopedia of Integer Sequences, published electronically at
  https://oeis.org. 

\bibitem{braunbruegge2022facets}
Benjamin Braun and Kaitlin Bruegge.
\newblock Facets of symmetric edge polytopes for graphs with few edges.
\newblock {\em Journal of Integer Sequences}, 26:Article 23.7.2, 2023.

\bibitem{normaliz}
W.~Bruns, C.~S\"oger B.~Ichim, and U.~von~der Ohe.
\newblock Normaliz. Algorithms for rational cones and affine monoids.
\newblock Available at https://normaliz.uos.de.

\bibitem{pyramiddecomp}
Winfried Bruns, Bogdan Ichim, and Christof S\"{o}ger.
\newblock The power of pyramid decomposition in {N}ormaliz.
\newblock {\em J. Symbolic Comput.}, 74:513--536, 2016.

\bibitem{chendavispersonal}
Tianran Chen and Robert Davis.
\newblock personal communication, 2021.

\bibitem{chen2021facets}
Tianran Chen, Robert Davis, and Evgeniia Korchevskaia.
\newblock Facets and facet subgraphs of symmetric edge polytopes.
\newblock {\em Discrete Applied Mathematics}, 328:139--153, 2023.

\bibitem{chen2017counting}
Tianran Chen, Robert Davis, and Dhagash Mehta.
\newblock Counting equilibria of the {K}uramoto model using birationally
  invariant intersection index.
\newblock {\em SIAM J. Appl. Algebra Geom.}, 2(4):489--507, 2018.

\bibitem{chen2020graphadjacency}
Tianran Chen and Evgeniia Korchevskaia.
\newblock Graph edge contraction and subdivisions for adjacency polytopes,
  2020.
\newblock preprint at https://arxiv.org/abs/1912.02841.

\bibitem{dali2022gammavector}
Alessio D'Alì, Martina Juhnke-Kubitzke, Daniel Köhne, and Lorenzo Venturello.
\newblock On the gamma-vector of symmetric edge polytopes.
\newblock {\em SIAM J. Discrete Math.}, 37(2):487--515, 2023.

\bibitem{dalidelucchimichalek}
Alessio D'Al\`\i, Emanuele Delucchi, and Mateusz Micha\l~ek.
\newblock Many faces of symmetric edge polytopes.
\newblock {\em Electron. J. Combin.}, 29(3):Paper No. 3.24, 42, 2022.

\bibitem{erdosrenyi1960}
P.~Erd\H{o}s and A.~R\'{e}nyi.
\newblock On the evolution of random graphs.
\newblock {\em Magyar Tud. Akad. Mat. Kutat\'{o} Int. K\"{o}zl.}, 5:17--61,
  1960.

\bibitem{SIAMMCMC}
Bailey~K. Fosdick, Daniel~B. Larremore, Joel Nishimura, and Johan Ugander.
\newblock Configuring random graph models with fixed degree sequences.
\newblock {\em SIAM Review}, 60(2):315--355, 2018.

\bibitem{Havel-Hakimi}
S.~L. Hakimi.
\newblock On realizability of a set of integers as degrees of the vertices of a
  linear graph ii. uniqueness.
\newblock {\em Journal of the Society for Industrial and Applied Mathematics},
  11(1):135--147, 1963.

\bibitem{higashitanifanopolytopes}
Akihiro Higashitani.
\newblock Smooth fano polytopes arising from finite directed graphs.
\newblock {\em Kyoto Journal of Mathematics}, 55(3), Sep 2015.

\bibitem{higashitanijochemkomateusz}
Akihiro Higashitani, Katharina Jochemko, and Mateusz Micha{\l}ek.
\newblock Arithmetic aspects of symmetric edge polytopes.
\newblock {\em Mathematika}, 65(3):763--784, 2019.

\bibitem{kalman2022ehrhart}
Tamás Kálmán and Lilla Tóthmérész.
\newblock Ehrhart theory of symmetric edge polytopes via ribbon structures,
  2022.
\newblock https://arxiv.org/abs/2201.10501.

\bibitem{kalman2022h*}
Tamás Kálmán and Lilla Tóthmérész.
\newblock $h^*$-vectors of graph polytopes using activities of dissecting
  spanning trees, 2022.
\newblock https://arxiv.org/abs/2203.17127.

\bibitem{matsuietal2011}
Tetsushi Matsui, Akihiro Higashitani, Yuuki Nagazawa, Hidefumi Ohsugi, and
  Takayuki Hibi.
\newblock Roots of ehrhart polynomials arising from graphs.
\newblock {\em Journal of Algebraic Combinatorics}, 34(4):721–749, May 2011.

\bibitem{newmannetworks}
Mark Newman.
\newblock {\em Networks}.
\newblock Oxford University Press, Oxford, 2018.

\bibitem{smoothfanoehrhart2012}
Hidefumi Ohsugi and Kazuki Shibata.
\newblock Smooth {F}ano polytopes whose {E}hrhart polynomial has a root with
  large real part.
\newblock {\em Discrete Comput. Geom.}, 47(3):624--628, 2012.

\bibitem{osughitsuchiya2020}
Hidefumi Ohsugi and Akiyoshi Tsuchiya.
\newblock The $h^*$-polynomials of locally anti-blocking lattice polytopes and
  their $\gamma $-positivity.
\newblock {\em Discrete \& Computational Geometry}, 66(2):701–722, Aug 2020.

\bibitem{symmetricedgematchingpolys}
Hidefumi Ohsugi and Akiyoshi Tsuchiya.
\newblock Symmetric edge polytopes and matching generating polynomials.
\newblock {\em Combinatorial Theory}, 1(0), Dec 2021.

\bibitem{convexhullcomputations}
Raimund Seidel.
\newblock Convex hull computations.
\newblock In {\em Handbook of discrete and computational geometry}, CRC Press
  Ser. Discrete Math. Appl., pages 361--375. CRC, Boca Raton, FL, 1997.

\bibitem{sage}
{The Sage Developers}.
\newblock {\em {S}ageMath, the {S}age {M}athematics {S}oftware {S}ystem
  ({V}ersion 9.3)}, 2021.
\newblock {\tt https://www.sagemath.org}.

\bibitem{wattsstrogatz1998}
Duncan~J Watts and Steven~H Strogatz.
\newblock Collective dynamics of ‘small-world’networks.
\newblock {\em {N}ature}, 393(6684):440--442, 1998.

\end{thebibliography}

\label{sec:biblio}

\end{document}